\documentclass[preprint,12pt]{elsarticle}

\usepackage{amssymb}
\usepackage{amsmath}
\usepackage{amsfonts}
\usepackage{amsthm}
\usepackage{algorithm}
\usepackage{algpseudocode}
\usepackage{booktabs}
\usepackage{subcaption}
\usepackage{longtable}
\usepackage{multirow}
\usepackage{array}
\usepackage{dsfont}
\usepackage{url}
\usepackage{float}
\usepackage{hyperref}
\usepackage{comment}
\usepackage{doi}

\newtheorem{proposition}{Proposition}[section]
\newtheorem{remark}{Remark}[section]

\journal{EURO Journal on Computational Optimization}

\begin{document}

\begin{frontmatter}

\title{Data-Driven Stochastic VRP: Integration of Forecast Duration into Optimization for Utility Workforce Management}

\author[univ]{Matteo Garbelli\corref{cor1}}
\ead{matteo.garbelli@univr.it}

\cortext[cor1]{Corresponding author}

\affiliation[univ]{organization={Department of Computer Science},
                   addressline={Strada le Grazie, 15},
                   city={Verona},
                   postcode={37134},
                   country={Italy}}

\begin{abstract}
This paper investigates the integration of machine learning forecasts of intervention durations into a stochastic variant of the Capacitated Vehicle Routing Problem with Time Windows (CVRPTW). In particular, we exploit tree-based gradient boosting (XGBoost) trained on eight years of gas meter maintenance data to produce point predictions and uncertainty estimates, which then drive a multi-objective evolutionary optimization routine. The methodology addresses uncertainty through sub-Gaussian concentration bounds for route-level risk buffers and explicitly accounts for competing operational KPIs through a multi-objective formulation. Empirical analysis of prediction residuals validates the sub-Gaussian assumption underlying the risk model. From an empirical point of view, our results report improvements around 20-25\% in operator utilization and completion rates compared with plans computed using default durations. The integration of uncertainty quantification and risk-aware optimization provides a practical framework for handling stochastic service durations in real-world routing applications.
\end{abstract}





\end{frontmatter}

\section{Introduction}

The Vehicle Routing Problem (VRP), first formalized in \cite{danzig}, addresses the efficient assignment of vehicles to serve a predetermined set of customers or locations and constitutes a fundamental optimization challenge in logistics, transportation, and utility services. Classical VRP methods assume complete information of travel times or fixed service durations and deterministic customer demands. While being mathematically tractable, this deterministic perspective poorly represents real operations and solutions often fail when deployed, becoming infeasible or highly suboptimal. In practice, traffic conditions fluctuate, service durations vary across different contexts, and customer demands change.

The gap between practice and theory has driven the development of a new class of problems, specifically the Stochastic VRPs, which explicitly model uncertainty. Recently, the inclusion of data-driven frameworks, such as predictive algorithms, can improve the performance and refine the efficiency of the route-generation process for optimization tasks.


This work addresses the challenge of incorporating stochastic intervention durations into predictive models for VRP optimization. We evaluate whether machine learning forecasts improve routing solutions compared to traditional fixed-time estimates used in practice.

Our study has three objectives: (1) develop a framework integrating ML forecasts into SVRP solvers; (2) evaluate multiple prediction models trained over eight years of operational data; (3) identify which forecast-optimization combinations perform best for practical deployment.

\paragraph{Organization of the paper}

The remainder of this paper is organized as follows. Section \ref{sec:svrp} provides an overview of Stochastic VRP and its variants with a glance towards optimization under uncertainty. Section \ref{sec:opt_methods} presents the chosen optimization algorithm to solve VRPs.
Section \ref{sec:predictive} discusses predictive methods for informed optimization. Section \ref{sec:experimental} presents the experimental framework developed for analyzing the performance of the forecast with respect to the actual baseline used for production. Section \ref{sec:vrp} focuses on producing representative simulations for multiple instances of VRPs, showing the potential and the efficiency of forecast integration. Finally, Section \ref{sec:conclusion} concludes the paper and outlines directions for future research.

\bigskip

\section{Literature Review}
\label{sec:svrp}

This section reviews best practices for stochastic VRPs, analyzing both classical approaches and the integration of machine learning methods. Traditionally, two principal methodological frameworks have emerged for effectively incorporating travel time uncertainty into VRP models with time windows: stochastic programming and robust optimisation.
The stochastic programming framework operates under the assumption that complete probability distributions of travel times are available. This methodology explicitly incorporates these distributions into the optimization model, typically through scenario-based formulations or chance constraints. Under this paradigm, travel times are modeled as random variables drawn from predetermined probability distributions, allowing for explicit consideration of uncertainty in the decision-making process. Classical surveys, such as \cite{Sahinidis2004, Shapiro2008}, offer a comprehensive overview of this approach.

Conversely, the robust optimization approach does not require the knowledge of specific probability distributions; rather, it is built over the construction of uncertainty sets. This methodology characterizes travel time variability through bounded intervals rather than complete distributional information. The objective is to develop solutions that maintain feasibility and near-optimal performance across all possible realizations within the specified uncertainty set.
We refer to \cite{delage2010distributionally} for a complete survey on this strand. We also mention another complementary approach, dealing with
contextual optimization \cite{CSO}, that uses local context variables (e.g., traffic information over some edges) to condition routing decisions.

The remainder of the section presents a thorough review of existing empirical approaches, mostly building on heuristic methods to model stochastic demand, time-dependent travel uncertainties, and multi-objective optimization, establishing the current state-of-the-art in SVRP research.

The comprehensive analysis by Oyola et al. \cite{Oyola1, Oyola2} demonstrates that uncertainty in VRP manifests across multiple operational dimensions, primarily affecting demand patterns, travel times, and service durations. The evolution of the deterministic capacitated VRP has generated multiple specialized variants designed to address specific operational constraints and real-world complexities. These variants include: 
\begin{enumerate}

\item the Capacitated Vehicle Routing Problem with Stochastic Demand (CVRPSD);

\item the Vehicle Routing Problem with Time Windows (VRPTW) incorporating temporal constraints for customer service;

\item the VRP with Pickups and Deliveries (VRPPD) addressing simultaneous collection and distribution requirements;

\item the VRP with Backhauling (VRPB), optimizing routes that combine delivery and subsequent collection operations. 
\end{enumerate}
Among these variants, the CVRPSD represents the most extensively investigated stochastic formulation, characterized by customer demands that follow probabilistic distributions and are revealed only after the initial route planning is completed.

The treatment of correlated stochastic travel times has received significant attention through the work of Rajabi-Bahaabadi et al. \cite{bahaabadi-et-al}, who address the reliable vehicle routing problem in stochastic networks with correlated travel times. A genetic evolutionary algorithm is proposed for modeling stochastic travel times with soft time windows and correlated stochastic arc travel times, which is simultaneously coupled with an ant colony optimization algorithm for solving the resulting stochastic programming problem.  The methodology derives closed-form expressions for calculating expected earliness and tardiness penalties while applying the Akaike Information Criterion to identify optimal probability distributions for travel time modeling. Shifted log-normal distributions demonstrate superior performance.

In the paradigm of contextual optimization, the VRPTW variant has been studied in the work presented in \cite{Serrano-et-al}. Their work investigates VRPTW under stochastic travel times where decision-makers observe contextual information represented as feature variables prior to making routing decisions. The approach integrates contextual feature variables directly into the optimization algorithm, minimizing total transportation costs and expected late arrival penalties conditioned on observed features. This methodology differs from previous approaches, e.g., \cite{Bomboi-et-al} and \cite{bahaabadi-et-al}, since a fixed known travel time distribution is assumed. Instead, \cite{Serrano-et-al} addresses scenarios where the joint distribution of travel times and features remains unknown, developing novel data-driven prescriptive models that utilize historical data to provide approximate solutions. Their data-driven prescriptive model connects to decision-makers with access to historical travel times for recent periods, enabling direct penalty prediction rather than requiring estimation of travel time distributions followed by expected penalty approximation.

Feasibility assessment for stochastic VRPTW has been advanced in \cite{Bomboi-et-al}, where efficient feasibility assessment methods for stochastic VRPTW with correlated and time-dependent travel times are developed. Their contributions include an approximate feasibility verification approach for single-chance constrained routing problems that concurrently considers travel time correlations and time dependencies. The methodology adopts a jointly normal distribution assumption for travel time modeling and establishes infeasibility criteria based on excessive time window violation probabilities.

For an environmental perspective, we refer \cite{Gulmez-et-al}, who extend stochastic formulations to include environmental factors through the Green VRP. Their approach implements flexible time windows to enhance both routing efficiency and environmental performance, develops a preference-based system allowing customers to rank alternative time windows, and formulates a multi-objective optimization model that balances operational costs, fossil fuel consumption, and customer satisfaction. The research creates realistic benchmark problems for performance evaluation and conducts comparative analysis of four evolutionary algorithm solvers (NSGA-II, NSGA-III, MOEA/D, and SMS-EMOA) for Pareto front approximation.

\bigskip

In practice, however, neither the stochastic programming nor the robust optimization paradigm provides a straightforward, scalable deployment for production systems. On the one hand, the stochastic approach demands high-fidelity probability distributions \cite{CSO} that are rarely available in real time. Conversely, the robust approach can become overly conservative when uncertainty sets are widened to cover edge cases, leading to infeasible routes or excessive buffer times.

To circumvent these limitations, a new class of data-driven and machine-learning approaches has been developed for directly solving or integrating the VRP.  As clearly illustrated in \cite{Bai2021, Bogyrbayeva2024,Shahbazian2024}, recent efforts have explored how learning paradigms—ranging from supervised and reinforcement learning to hybrid neuro-heuristic frameworks - can replace or augment traditional optimization pipelines for vehicle routing problems.
\cite{Bogyrbayeva2024} provides a detailed classification that differentiates between end-to-end learning, hybrid learning combined with heuristics, single-agent and multi-agent formulations. It shows that ML algorithms and architectures can produce near-optimal solutions while operating significantly faster.
\cite{Shahbazian2024} extends this picture with an application-oriented survey, showing that attention-based networks and graph neural encoders now outperform classical meta-heuristics on large-scale, dynamic VRP variants: modern ML models can reliably capture stochastic travel times, forecast customer demand, learn adaptive routing policies, and plug directly into optimizers. Also \cite{Bai2021} offers a practitioner-oriented review that maps out how analytics and ML can be woven into every stage of a VRP workflow: stochastic-demand forecasting, parameter estimation, model selection, decomposition guidance, and neighbourhood search.

A potential strategy is to develop Reinforcement Learning strategies that learn directly from historical operational data without requiring explicit probability distributions. Deep Reinforcement Learning (DRL) approaches \cite{Nazari2018, Raza2022, Pan2023} have demonstrated particular promise in addressing stochastic VRP via interaction with simulated environments. These methods formulate routing decisions as a Markov Decision Process, enabling dynamic response to uncertain conditions such as stochastic travel times, service times, and customer demands. An interesting approach based on attention-driven transformers that capture circularity and symmetry is proposed \cite{Guan2025}. These methods effectively manage time uncertainty, also in a framework where correlations among duration times are present \cite{Iklassov2023}. Correlated stochastic demands are also handled in \cite{Florio2023}, which presents a Bayesian learning framework integrated with an elementary branch-price-and-cut algorithm. 

The other possible framework deals with the integration of forecast algorithms. Rather than modeling uncertainty through probability distributions, these approaches utilize machine learning models to directly predict uncertain parameters from historical data. \cite{Chang2024} develops a predict-then-optimize framework that integrates graph neural networks with combinatorial optimization for vehicle routing under uncertainty.

Hybrid methodologies that combine forecasting with optimization have shown particular effectiveness for SVRPs with stochastic service times. For example, \cite{Rezvanian2025} develops a Grouping Evolution Strategy (GES) algorithm. The proposed GES algorithm effectively handles the problem's complexity by integrating specialized solution representation, mutation operators, and local search heuristics.

As already emphasized, these data-driven paradigms offer significant advantages over traditional approaches: they eliminate the need for explicit probability distributions, adapt to real-time conditions through continuous learning, and leverage historical patterns to inform routing decisions. As noted in the comprehensive review \cite{Soeffker2022}, ML methods are increasingly bridging the gap between theoretical optimization models and practical implementation in evolving urban logistics environments.

\medskip

Hence, in the current paper, we adopt this perspective and we develop a data-driven workflow that trains an XGBoost model on historical service-duration records spanning several years to predict the actual time required for each planned activity. The resulting point forecasts, combined with calibrated prediction intervals, are fed directly into the VRP solver.

\section{Solving the VRP via Evolutionary Algorithm}
\label{sec:opt_methods}

The Vehicle Routing Problem can be formulated on a graph $G = (V, A)$, where $V$ represents nodes corresponding to activities and $A$ represents feasible connections between these activities.

\begin{table}[H]
\centering
\caption{Notation.}
\label{tab:notation}
\begin{tabular}{ll}
\toprule
Symbol & Meaning \\
\midrule
$G=(V,A)$ & Directed graph (depot $0$, customers $N$) \\
$K$ & Set of vehicles, $|K|=m$ \\
$c_{ij},\, t_{ij}$ & Travel cost/time (min) from $i$ to $j$ \\
$[a_i,b_i]$ & Time window at customer $i$ \\
$H_k$ & Max duration (shift length, min) for vehicle $k$ \\
$S_i$ & Random service duration (min) at $i$ \\
$\mu_i$ & Predicted service duration (point forecast, min) \\
$\sigma_i^2$ & Residual proxy variance (sub-Gaussian) \\
$U_i$ & One-sided conformal upper width at level $1-\tilde{\alpha}$ \\
$x_{ijk}$ & Binary: vehicle $k$ uses arc $(i,j)$ \\
$T_{ik}$ & Planned start time (min) of service at $i$ by $k$ \\
$\delta_i$ & Tardiness slack (min) for soft time windows \\
\bottomrule
\end{tabular}
\end{table}

  \subsection{The Capacitated VRP with uncertain activity durations}
  \label{cvprtw}

Let $G=(V,A)$ be a directed graph, where the vertex set $V=\{0\} \cup N \cup \{n+1\}$ comprises the starting depot $0$, the ending depot $n+1$, and the set of customers $N=\{1, \dots, n\}$. The fleet consists of vehicles $K=\{1, \dots, m\}$. We assume deterministic travel times $t_{ij} \ge 0$ (in minutes) and travel costs $c_{ij} \ge 0$. Furthermore, each customer $i$ is associated with a time window $[a_i, b_i]$, and each vehicle $k$ is constrained by a maximum shift duration $H_k$.

Service duration at $i$ is random: $S_i=\mu_i+\varepsilon_i$, where $\mu_i=\mu(x_i)$ is the ML prediction from features $x_i$, and $\varepsilon_i$ is zero-mean noise. A plan consists of binary routing variables $x_{ijk}\in\{0,1\}$ and continuous start times $T_{ik}\ge0$ (min).

The stochastic CVRPTW with chance constraints is formulated as:

\begin{align}
\min\ & \sum_{k\in K}\sum_{i,j\in V} c_{ij} x_{ijk}
      + \lambda \sum_{i\in N} \delta_i \label{eq:objective_stoch}\\
\text{s.t. }
& \sum_{j\in V} x_{0jk}=1,\ \sum_{i\in V} x_{i,n+1,k}=1 && \forall k\in K \label{eq:depot}\\
& \sum_{j\in V} x_{jik}=\sum_{j\in V} x_{ijk} && \forall i\in N,\,\forall k\in K \label{eq:flow}\\
& \sum_{k\in K}\sum_{j\in V} x_{ijk}=1 && \forall i\in N \label{eq:visit}\\
& T_{jk} \ge T_{ik} + \mu_i + t_{ij} - M(1-x_{ijk}) && \forall i,j\in V,\ \forall k\in K \label{eq:time_prop}\\
& a_i \le T_{ik} \le b_i + \delta_i,\ \delta_i\ge0 && \forall i\in N,\ \forall k\in K \label{eq:time_window}\\
& \mathbb{P}\!\left\{T_{n+1,k}-T_{0k} \le H_k\right\} \ge 1-\alpha_k && \forall k\in K \label{eq:chance}\\
& x_{ijk}\in\{0,1\},\ T_{ik}\ge0. \label{eq:binary}
\end{align}

The objective function (\ref{eq:objective_stoch}) minimizes travel costs $c_{ij}$ and penalizes tardiness $\delta_i$ via weight $\lambda$. Constraint (\ref{eq:depot}) ensures each vehicle leaves and returns to the depot. Constraint (\ref{eq:flow}) maintains flow consistency. Constraint (\ref{eq:visit}) ensures each customer is visited exactly once. Constraint (\ref{eq:time_prop}) propagates time using big-$M$ formulation, where $M$ is a sufficiently large constant. Constraint (\ref{eq:time_window}) enforces time windows with soft tardiness slack $\delta_i$. Critically, constraint (\ref{eq:chance}) is a \emph{chance constraint} ensuring that route duration stays within shift length $H_k$ with probability at least $1-\alpha_k$. The decision variables $x_{ijk}$ and $T_{ik}$ are deterministic; the challenge lies in satisfying the chance constraint (\ref{eq:chance}), which depends on the sum of random service durations $S_i$ along each route. To handle this, we need a tractable way to model the uncertainty from the prediction errors, $\varepsilon_i$.

We model the prediction residuals $\varepsilon_i$ as independent sub-Gaussian random variables. This assumption is a powerful generalization of the Gaussian distribution, providing strong guarantees for the sum of variables, even when their individual distributions are not perfectly normal. Its practical importance lies in the fact that it permits controlling the risk of an entire route failing to meet its time constraint without needing to know the exact distribution of each service time. While the underlying true duration distributions show some skewness (see Figure~\ref{fig:distributions}), the sub-Gaussian assumption on the model residuals is reasonable, as we will demonstrate in Section~\ref{sec:predictive} (see Figure~\ref{fig:error_dist}).

Assuming the residuals $\varepsilon_i$ are independent and sub-Gaussian with proxy variance $\sigma_i^2$ (estimated from model residuals), their sum concentrates tightly around its mean. For a route $R$ and any $u>0$, we have the concentration inequality:
\begin{equation}
\mathbb{P}\!\left\{\sum_{i\in R}\varepsilon_i \ge u\right\} \le \exp\!\left(-\frac{u^2}{2\sum_{i\in R}\sigma_i^2}\right).
\end{equation}
Setting the RHS to $\alpha$ yields a route-level buffer:
\begin{equation}\label{eq:sg-buffer}
\Delta_\alpha(R)\ :=\ \sqrt{2\left(\sum_{i\in R}\sigma_i^2\right)\log(1/\alpha)}.
\end{equation}

We can then state a sufficient condition for chance feasibility:

\begin{proposition}[Route-level chance feasibility via sub-Gaussian buffer]
\label{prop3.1}
Let $R_k$ be the sequence of customers on vehicle $k$'s route and suppose $S_i=\mu_i+\varepsilon_i$ with independent sub-Gaussian $\varepsilon_i$ having proxy variances $\sigma_i^2$. If
\[
\sum_{i\in R_k}\mu_i + \sum_{(i,j)\in R_k} t_{ij}
+ \sqrt{2\log(1/\alpha_k)\sum_{i\in R_k}\sigma_i^2}
\ \le\ H_k,
\]
then $\mathbb{P}\{T_{n+1,k}-T_{0k} \le H_k\}\ge 1-\alpha_k$.
\end{proposition}

\begin{proof}
    By sub-Gaussian concentration, the sum of independent sub-Gaussian random variables $\sum_{i\in R_k}\varepsilon_i$ satisfies the tail bound above. The route duration is $\sum_{i\in R_k}(\mu_i+\varepsilon_i) + \sum_{(i,j)\in R_k} t_{ij}$. The deterministic part is $\sum_{i\in R_k}\mu_i + \sum_{(i,j)\in R_k} t_{ij}$. Adding the buffer $\Delta_{\alpha_k}(R_k)$ ensures that with probability at least $1-\alpha_k$, the stochastic deviations $\sum_{i\in R_k}\varepsilon_i$ do not exceed the buffer, hence the total route duration stays within $H_k$. \hfill$\square$
\end{proof}


\begin{remark}
As an alternative approach, one can use conformal prediction intervals. Suppose one-sided $(1-\tilde{\alpha})$ conformal upper bounds $U_i$ satisfy $\mathbb{P}\{S_i \le \mu_i + U_i\} \ge 1-\tilde{\alpha}$ for all $i$. Using Bonferroni union bound with $\tilde{\alpha}=\alpha_k/|R_k|$, we obtain:
\begin{equation}
\mathbb{P}\left\{\sum_{i\in R_k} S_i \le \sum_{i\in R_k} (\mu_i + U_i) \right\}
\ \ge\ 1-\alpha_k.
\end{equation}
This provides distribution-free coverage guarantees without requiring sub-Gaussian assumptions, albeit at the cost of potentially wider intervals due to the Bonferroni correction. While this provides a robust distribution-free guarantee, the union bound can be overly conservative for routes with many activities. For practical planning systems, integrating point forecasts with a single, aggregate risk buffer, as enabled by the sub-Gaussian approach, is often more tractable. Therefore, we adopt the latter for our main analysis.
\end{remark}

\subsection{The optimization algorithm}

The evolutionary approach addresses the multi-objective nature of the Stochastic VRP with stochastic durations by simultaneously optimizing conflicting objectives. Real-world routing operations involve competing KPIs: minimizing travel costs, reducing overtime, meeting time windows, and maximizing service coverage.

We formulate the CPVRTW introduced in \ref{cvprtw} as a multi-objective optimization problem where the objective vector is:
\begin{equation}\label{eq:multi_obj}
\min\left( \underbrace{\sum_{k}\sum_{i,j} c_{ij}x_{ijk}}_{\text{travel cost}},\ \underbrace{\sum_i \delta_i}_{\text{time-window violations}},\ \underbrace{\sum_k \mathrm{Overtime}_k}_{\text{overtime}},\ \underbrace{-\#\text{served}}_{\text{maximize coverage}} \right).
\end{equation}
Here $\mathrm{Overtime}_k$ represents the amount of time vehicle $k$ exceeds shift length $H_k$, computed as $\max(0, T_{n+1,k}-T_{0k}-H_k)$. The negative sign on served tasks converts maximization to minimization. Rather than scalarizing via fixed weights, we employ a Multi-Objective Evolutionary Algorithm (MOEA), specifically NSGA-III \cite{nsga3}, to explore the Pareto front and provide decision-makers with diverse trade-off solutions.

In a more general form, the multi-objective problem is:
  \begin{align}
  \min_{\mathbf{x} \in \Omega} \mathbf{F}(\mathbf{x}) = (f_1(\mathbf{x}), f_2(\mathbf{x}), \ldots, f_m(\mathbf{x}))
  \end{align}
  where $\mathbf{x}$ represents a routing solution, $\Omega$ is the feasible solution space, and $\mathbf{F}(\mathbf{x})$ is the vector of $m$ objective
  functions as specified in Eq.~(\ref{eq:multi_obj}).

  The algorithm maintains a population $P_t$ of $\mu$ individuals at generation $t$, where each individual represents a complete routing solution.
  Solutions are ranked using Pareto dominance: solution $\mathbf{x}_1$ dominates $\mathbf{x}_2$ (denoted $\mathbf{x}_1 \prec \mathbf{x}_2$) if:
  \begin{align}
  \forall i \in \{1,2,\ldots,m\}: f_i(\mathbf{x}_1) \leq f_i(\mathbf{x}_2) \quad \text{and} \quad \exists j \in \{1,2,\ldots,m\}: f_j(\mathbf{x}_1) <
  f_j(\mathbf{x}_2)
  \end{align}
  fronts contain progressively dominated solutions.

  Structured reference points in $\mathbf{R} = \{\mathbf{r}_1, \mathbf{r}_2, \ldots, \mathbf{r}_H\}$ have been utilized by the algorithm for the diversity of solutions in the objective space. These reference points are adaptively updated based on the current population distribution. For each reference point $\mathbf{r}_i$, the algorithm computes the perpendicular distance $d(\mathbf{x}, \mathbf{r}_i)$ from each solution $\mathbf{x}$, to create associations.

  The adaptation mechanism generates additional reference points around existing ones using:
  \begin{align}
  \mathbf{r}_{new} = \mathbf{r}_i + \Delta \mathbf{r}
  \end{align}
  where $\Delta \mathbf{r}$ represents adaptive perturbations based on population density in different objective space regions.

  Parent selection employs tournament selection with tournament size $\tau$, selecting parents $\mathcal{P}$ from population $P_t$. Crossover operations
  generate offspring $\mathcal{O}$ with probability $p_c$, while mutation operators introduce genetic diversity with probability $p_m$. Each offspring is
  evaluated against all objective functions, accounting for stochastic parameter realisations through Monte Carlo sampling or scenario-based approaches. We present the algorithm \ref{alg:moea-svrp}, which is based on \cite{nsga3, nsga3part2}.

 \begin{algorithm}[H]
  \caption{Forecast-aware MOEA for Stochastic CVRPTW}
  \label{alg:moea-svrp}
  \begin{algorithmic}[1]
  \State \textbf{Input:} Instance $(G,t,c,[a,b],H)$; ML model $\widehat{\mu}$ for predictions $\mu_i$; uncertainty proxy (per-stop $\sigma_i$ or conformal $U_i$); MOEA parameters; risk level $\alpha$
  \State \textbf{Initialize} population $P_0$ by greedy time-window insertion using predicted durations $\mu_i$
  \State $t \leftarrow 0$
  \While{termination criteria not satisfied}
      \State $t \leftarrow t + 1$
      \State \textbf{Variation:} Apply crossover \& mutation to $P_t$ preserving time-window feasibility when possible
      \State \textbf{Repair:} Local search (relocate/swap/2-opt) on offspring $\mathcal{O}$ until no improving move
      \State \textbf{Risk buffer:} For each route $R$ in $\mathcal{O}$, compute buffer $\Delta_\alpha(R)$ via Eq.~(\ref{eq:sg-buffer}) or conformal method
      \State \textbf{Feasibility check:} Penalize routes if $\sum_{i\in R}\mu_i+\sum_{(i,j)\in R} t_{ij}+\Delta_\alpha(R)>H$
      \State \textbf{Evaluate} multi-objective vector (Eq.~\ref{eq:multi_obj}): (travel, tardiness, overtime, $-$served)
      \State \textbf{Non-dominated Sorting:} Organize $P_t \cup \mathcal{O}$ into Pareto fronts $F_1, F_2, \ldots$
      \State \textbf{Select} $\mu$ individuals for $P_{t+1}$ via niching and crowding distance (NSGA-III)
      \State Update convergence and diversity metrics
  \EndWhile
  \State \textbf{return} Pareto set $P_G$; decision-maker chooses policy via KPI preferences
  \end{algorithmic}
  \end{algorithm}

The EA-based approach is particularly advantageous for stochastic VRP because it naturally maintains a diverse set of solutions, providing decision-makers with alternatives that represent different trade-offs between cost efficiency and robustness.
Moreover, it can easily incorporate complex objective functions and constraints without requiring mathematical reformulation and is less susceptible to getting trapped in local optima, which is especially important in stochastic environments where the fitness landscape can be highly non-convex and noisy.

\section{Predictive algorithms for informed optimization methods}
\label{sec:predictive}

This section presents the machine learning framework developed to address the problem of uncertainty in intervention times for workforce management. The forecast engine implements a modular architecture that integrates predictive algorithms with the stochastic VRP optimization pipeline introduced in the previous section.

\subsection{Mathematical Framework and Architecture Overview}

We now formalize the mathematical framework for duration prediction. We develop four model architectures addressing class imbalance, heterogeneous activity types, and multimodal duration distributions.

The forecast engine implements a unified mathematical framework for duration prediction. Let $(\Omega, \mathcal{F}, \mathbb{P})$ be a probability space representing the stochastic environment, and let $\mathcal{X} \subseteq \mathbb{R}^d$ denote the feature space with $\mathcal{Y} \subseteq \mathbb{R}_+$ representing the space of intervention durations. Let the target variable $Y$ correspond to the service duration $S$ defined in the VRP model in Section \ref{cvprtw}. The historical dataset $\mathcal{D} = \{(X_i, Y_i)\}_{i=1}^n$ consists of independent and identically distributed pairs where $X_i: \Omega \to \mathcal{X}$ are feature vectors and $Y_i: \Omega \to \mathcal{Y}$ are corresponding intervention durations.

The forecasting problem seeks to learn a predictive mapping $f^*: \mathcal{X} \to \mathcal{Y}$ that minimizes the expected risk under a specified loss function $\ell: \mathcal{Y} \times \mathcal{Y} \to \mathbb{R}_+$:
\begin{equation}
\label{forecast}
f^* = \arg\min_{f \in \mathcal{F}} \mathbb{E}_{(X,Y) \sim \mathbb{P}}[L (Y, f(X))]
\end{equation}
where $\mathcal{F}$ represents the hypothesis space of admissible functions.

The forecast engine transforms historical operational data into duration predictions for route optimization. The system has two phases: $(i)$ an offline training pipeline that learns patterns from historical data; $(ii)$ an online inference pipeline that generates predictions for daily planning. This separation enables exhaustive model optimization during training while preserving low-latency, i.e., fast inference at deployment.

The training pipeline (Figure~\ref{fig:train_pipeline}) processes historical data through three stages: data collection, feature engineering and model training with validation.

\begin{figure}[H]
    \centering
    \includegraphics[width=0.95\linewidth]{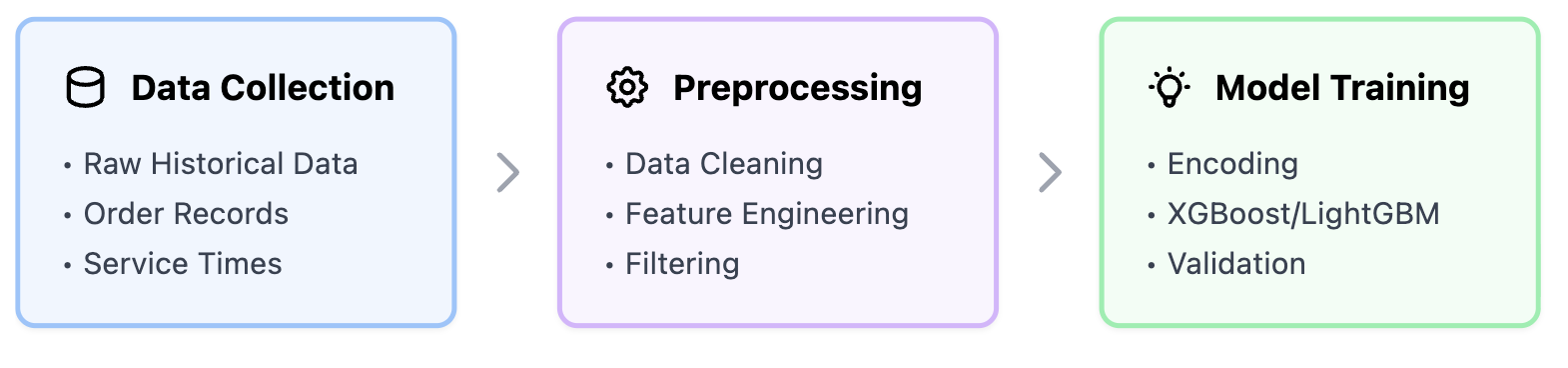}
    \caption{Training framework for duration prediction. 
    }
    \label{fig:train_pipeline}
\end{figure}

The {data preparation stage} validates data quality, detecting missing values, outliers, and inconsistencies. The dataset spans eight years and handles evolving activity categories and changing service areas.

The {feature engineering stage} transforms raw measurements into predictive features. Intervention durations depend on temporal patterns (hour-of-day, day-of-week, seasonality), geography (municipality, altitude, urbanization), and operational context (activity type, equipment, historical rates). We create interaction terms and polynomial features to capture these dependencies.

The {model estimation stage} trains four XGBoost architectures: standard, frequency-weighted (for class imbalance), dual (for heterogeneous activities), and dual-weighted (combined). We optimize hyperparameters using Bayesian optimization and validate using stratified k-fold cross-validation, temporal holdouts, and activity-specific performance analysis.

The inference framework (Figure~\ref{fig:inference_pipeline}) generates real-time duration predictions for VRP solvers, balancing speed and accuracy.

\begin{figure}[H]
    \centering
    \includegraphics[width=1.04\linewidth]{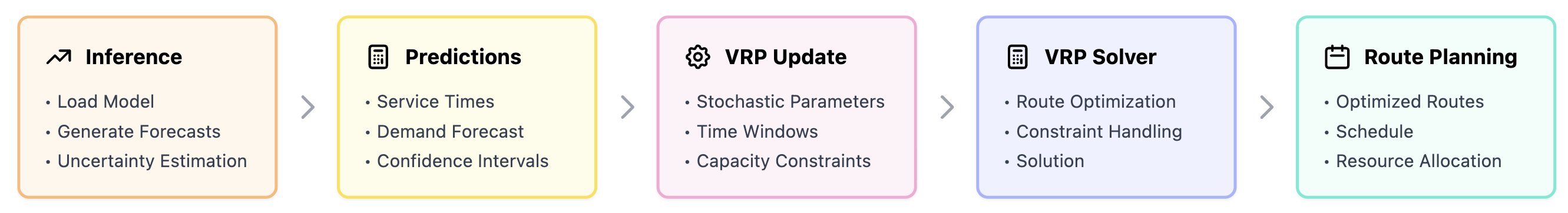}
    \caption{Inference framework for operational deployment. 
    }
    \label{fig:inference_pipeline}
\end{figure}

The {feature consistency module} applies the same transformations (cyclical encodings, interactions, normalizations) at inference as during training. It detects out-of-distribution inputs and handles missing values using validated imputation strategies.

The {model selection module} routes requests to specialized models: $f_Z$ handles Type-Z interventions (meter replacements), while $f_{\text{other}}$ handles standard activities.


\subsection{Model Architectures}

To approximate $f^*$ in Eq. \ref{forecast}, we try the following tailored gradient-boosting models:

\begin{enumerate}
    \item Standard Gradient Boosting Model
    \item Frequency-Weighted Gradient Boosting Model
    \item Dual Model Architecture
    \item Dual Models with Frequency Weighting
\end{enumerate}

\paragraph{1) Standard Gradient Boosting Model}

The baseline approach implements gradient boosting with uniform sample weighting. The model constructs an ensemble predictor through additive tree learning:
\begin{equation}
f_{\text{std}}(\mathbf{x}) = \sum_{m=1}^{M} \gamma_m \cdot T_m(\mathbf{x})
\end{equation}
where $M$ represents the total number of trees, $\gamma_m \in \mathbb{R}$ denotes the weight of tree $m$, and $T_m: \mathbb{R}^d \to \mathbb{R}$ represents the $m$-th decision tree mapping feature vectors to leaf values. The training process optimizes a regularized empirical risk functional:
\begin{equation}
\label{baseline_loss}
\mathcal{J}_{\text{std}}(\mathbf{f}) = \frac{1}{n} \sum_{i=1}^{n} L(y_i, f_{\text{std}}(\mathbf{x}_i)) + \lambda \sum_{m=1}^{M} \Phi(T_m)
\end{equation}
where $L: \mathbb{R} \times \mathbb{R} \to \mathbb{R}_+$ represents the loss function, $\lambda > 0$ controls regularization strength, and $\Phi(T_m)$ implements structural penalties on tree complexity, including leaf count and weight magnitudes.

\paragraph{2) Frequency-Weighted Gradient Boosting Model}

To address class imbalance across activity types, the framework implements sample-specific weighting based on activity type prevalence. The weighting mechanism assigns higher importance to underrepresented activity types through inverse frequency scaling:
\begin{equation}
\omega_i^{(c)} = \frac{n}{n_c \cdot |\mathcal{C}|}
\end{equation}
where $n$ represents total sample size, $n_c$ denotes samples from activity type $c$, and $|\mathcal{C}|$ represents the number of distinct activity types.

The weighted objective modifies the standard formulation by incorporating sample-specific importance:
\begin{equation}
\mathcal{J}_{\text{weighted}}(\mathbf{f}) = \frac{1}{n} \sum_{i=1}^{n} \omega_i^{(c_i)} \cdot L(y_i, f_{\text{weighted}}(\mathbf{x}_i)) + \lambda \sum_{m=1}^{M} \Phi(T_m)
\end{equation}
where $c_i$ denotes the activity type of sample $i$. This formulation ensures that rare activity types receive appropriate representation during model training despite their lower frequency in the operational data.

\subsubsection{Dual Models}

To account for the strong heterogeneity observed across activity types, as evidenced by the duration distributions in Figures~\ref{fig:violin}--\ref{fig:bins_norm}, separate predictive models were developed for specific categories of interventions. Precisely, we consider models that generalize Eq. \eqref{forecast} and that support heterogeneous modeling through domain decomposition.

Let $\{\mathcal{X}_k\}_{k=1}^K$ be a partition of the feature space such that $\bigcup_{k=1}^K \mathcal{X}_k = \mathcal{X}$ and $\mathcal{X}_i \cap \mathcal{X}_j = \emptyset$ for $i \neq j$. The heterogeneous predictor is then defined as:
\begin{equation}
f_{\text{het}}(\mathbf{x}) = \sum_{k=1}^K \mathds{1}_{x \in \mathcal{X}_k} \cdot f_k(\mathbf{x})
\end{equation}
where each $f_k: \mathcal{X}_k \to \mathcal{Y}$ is specialized for the patterns within subdomain $\mathcal{X}_k$. This decomposition enables the learning of specialized models that capture distinct operational regimes while maintaining global consistency.

\paragraph{3) Dual Models Architecture}

This approach implements a specialized dual architecture that trains separate models for distinct intervention categories. Building upon the heterogeneous predictor framework in Equation (18), we partition activities into two semantically meaningful classes based on the activity type feature $a(\mathbf{x})$:

\begin{equation}
f_{\text{dual}}(\mathbf{x}) = \begin{cases}
f_Z(\mathbf{x}) & \text{if } a(\mathbf{x}) = Z \\
f_{\text{other}}(\mathbf{x}) & \text{if } a(\mathbf{x}) \neq Z
\end{cases}
\end{equation}

where $a: \mathbb{R}^d \to \mathcal{C}$ represents the activity type classification function, $f_Z: \mathbb{R}^d \to \mathbb{R}_+$ specializes in Type-Z intervention prediction, and $f_{\text{other}}: \mathbb{R}^d \to \mathbb{R}_+$ handles all other activity types.

Each specialized component is trained using the standard gradient boosting objective (Eq. 15) but optimized for the distinct statistical characteristics of its operational domain. The Type-Z model emphasizes flexibility to capture high variability patterns through reduced regularization, while the other-types model prioritizes stability through enhanced regularization. This architectural separation enables each model to specialize in its respective operational regime while maintaining computational efficiency.

\paragraph{4) Dual Models with Frequency Weighting}

The most sophisticated approach combines architectural specialization with frequency-based weighting, addressing both structural heterogeneity and statistical imbalance. This represents a hierarchical extension of the dual architecture where each specialized component incorporates sample weighting:

\begin{equation}
\label{dual_weighted}
f_{\text{dual-weighted}}(\mathbf{x}) = \begin{cases}
f_{Z,w}(\mathbf{x}) & \text{if } a(\mathbf{x}) = Z \\
f_{\text{other},w}(\mathbf{x}) & \text{if } a(\mathbf{x}) \neq Z
\end{cases}
\end{equation}

Each component model implements frequency weighting computed within its respective operational domain using the weighting scheme from Equation (16). The Type-Z component $f_{Z,w}$ applies inverse frequency scaling based on subtype prevalence within Type-Z interventions, while the other-types component $f_{\text{other},w}$ computes weights relative to the distribution of non-Z activity types.

In summary, we present the key characteristics of the four model architectures in Table \ref{tab:model_summary}.

\begin{table}[H]
\centering
\caption{Summary of Model Architectures}
\label{tab:model_summary}
\begin{tabular}{lcc}
\hline
\textbf{} & \textbf{\# Models} & \textbf{Weighting} \\
\hline
Standard: Baseline  & 1 & \\
Frequency-Weighted: Class imbalance & 1 & \checkmark \\
Dual: Specialized models & 2 & \\
Dual-Weighted: Specialization + imbalance & 2 & \checkmark \\
\hline
\end{tabular}
\end{table}

A priori, developing separate models for each activity class might seem optimal, but analysis of duration distributions (Figure~\ref{fig:violin}) reveals that most non-Type-Z activities exhibit similar statistical properties. Moreover, training and maintaining numerous specialized models (beyond the dual approach) present operational challenges in production environments, including increased complexity and the related reduced reliability. Therefore, we select dual architecture as a pragmatic balance between specialization and operational feasibility.

\paragraph{Residual Analysis.} Recall that in Section~\ref{sec:opt_methods}, we assumed prediction residuals $\varepsilon_i$ follow a sub-Gaussian distribution to enable tractable route-level risk buffers. This assumption is empirically supported by the distribution of residuals from the Dual Weighted model. As shown in Figure~\ref{fig:error_dist}, the errors are centered at zero and exhibit a symmetric, bell-shaped distribution. The symmetric, bell-shaped distribution of residuals validates our assumption of sub-Gaussianity.

\begin{figure}[H]
    \centering
    \includegraphics[width=0.81\linewidth]{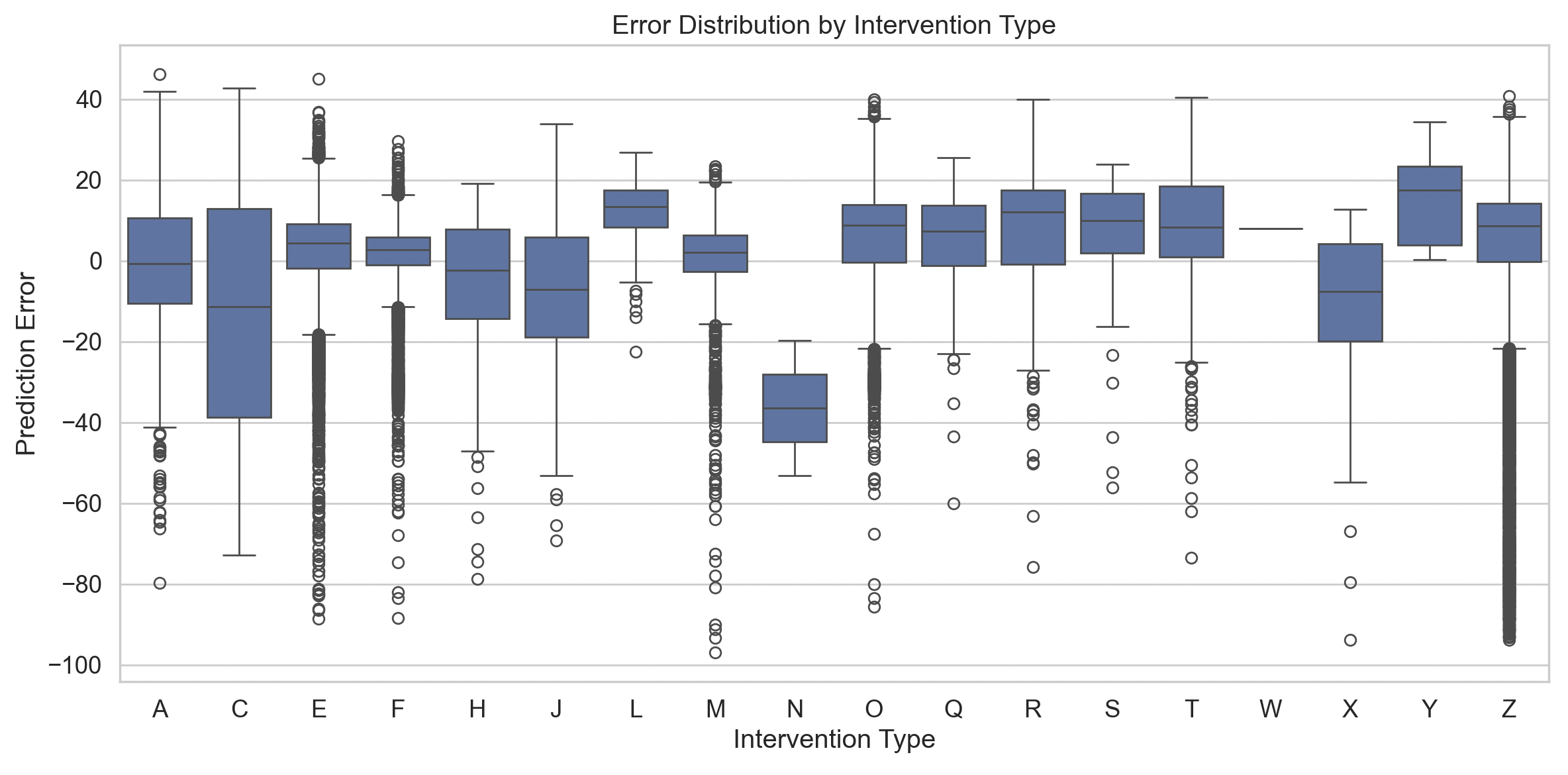}
    \caption{Distribution of prediction residuals ($\hat{y} - y$) from the Dual Weighted forecast model. The symmetric, bell-shaped curve supports the sub-Gaussian assumption used for risk modeling.}
    \label{fig:error_dist}
\end{figure}

\subsubsection{Estimating Proxy Variance for Risk Buffers}
\label{sec:proxy_variance}

While the XGBoost model provides point predictions $\mu_i$ representing the expected duration of an intervention, the robustness of the routing schedule relies on the sub-Gaussian risk buffers defined in Eq.~\eqref{eq:sg-buffer}. These buffers require an estimate of the proxy variance $\sigma_i^2$ for each activity, which quantifies the uncertainty of the prediction.

We estimate these variances using a residual-based approach stratified by activity type. This method leverages the historical prediction errors on the validation set to construct a lookup table of uncertainties. For each activity class $c \in \mathcal{C}$, the empirical proxy variance is computed as:
\begin{equation}
\label{eq:sigma_est}
\hat{\sigma}_c^2 = \frac{1}{n_c}\sum_{i \in \mathcal{V}_c} (y_i - \hat{\mu}_i)^2,
\end{equation}
where $\mathcal{V}_c$ is the set of validation samples of type $c$, and $n_c = |\mathcal{V}_c|$. This approach assumes that interventions of the same type share a similar error profile, an assumption supported by the residual analysis in Figure~\ref{fig:error_dist}.

\paragraph{Bridging Forecasts and Optimization}
This variance estimate serves as the critical link between the machine learning module and the stochastic optimization constraints (Proposition \ref{prop3.1}). In practice, the integration functions as a closed-loop system:

\begin{enumerate}
    \item For a new intervention $i$, the ML model predicts the expected duration $\mu_i$, while the uncertainty module assigns the corresponding variance $\sigma_{i}^2$ based on Eq.~\eqref{eq:sigma_est}.
    \item During the evolutionary search, for any candidate route $R$, the solver calculates the deterministic length ($\sum \mu_i + \sum t_{ij}$) and aggregates the uncertainty to compute the dynamic risk buffer $\Delta_\alpha(R)$ defined in Eq.~\eqref{eq:sg-buffer}:
    \[
    \Delta_\alpha(R) = \sqrt{2\log(1/\alpha)\sum_{i \in R}\sigma_i^2}.
    \]
    \item The route is deemed feasible only if the sum of the deterministic components and the risk buffer remains within the shift limit $H_k$, as per Proposition \ref{prop3.1}.
\end{enumerate}

This mechanism ensures adaptive conservatism. Interventions with high predictive uncertainty (e.g., Type-Z meter replacements, which exhibit high variance in residuals) contribute larger values to $\sum \sigma_i^2$, automatically generating a larger safety buffer $\Delta_\alpha(R)$. Conversely, routine activities with low variance result in tighter buffers, allowing for higher resource utilization. The parameter $\alpha_k$ acts as a "tuning knob" for calibration, allowing decision-makers to explicitly set the acceptable risk level (e.g., a 5\% probability of overtime) without requiring computationally expensive scenario sampling.

Figure~\ref{fig:predictions_with_ci} illustrates these estimates, showing how the confidence intervals (derived from $\sigma_i$) widen for complex activities and narrow for routine tasks.

\begin{figure}[H]
    \centering
    \includegraphics[width=0.85\linewidth]{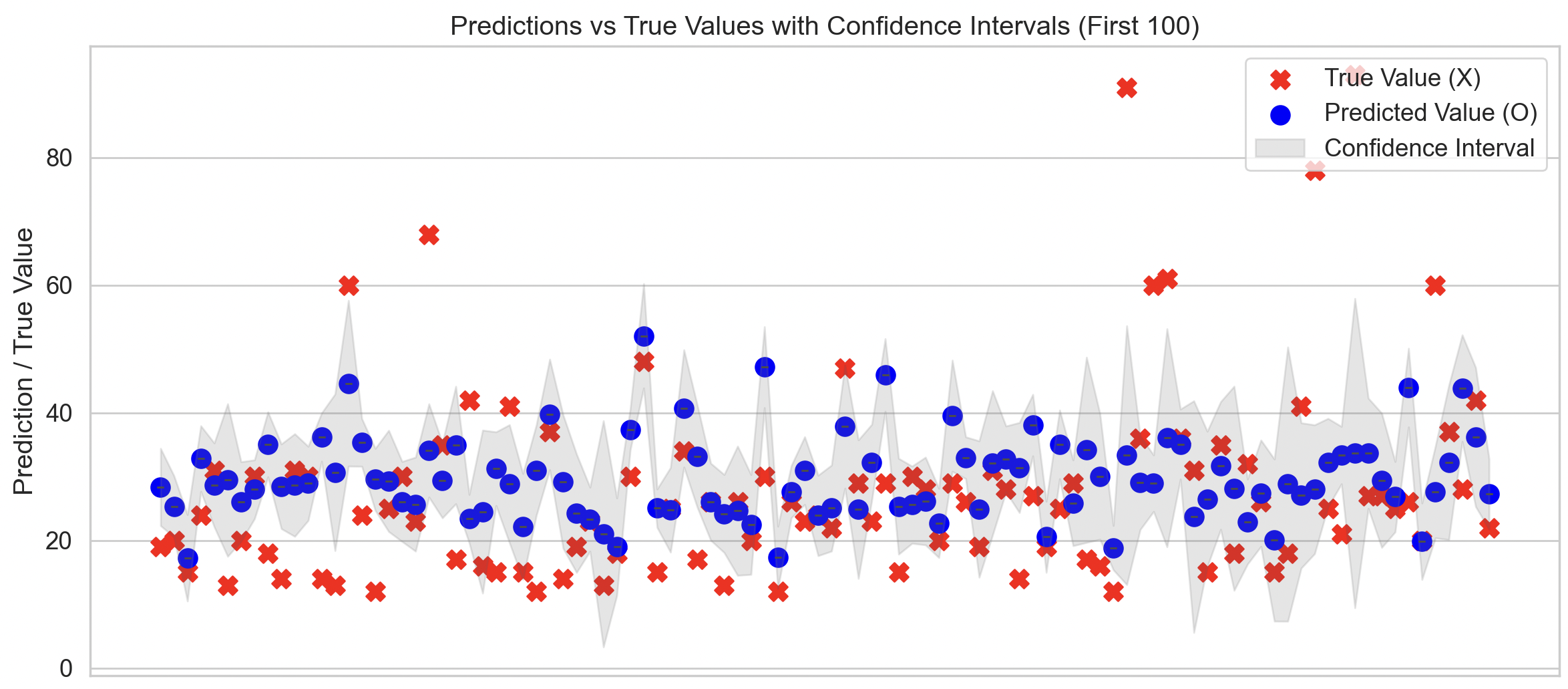}
    \caption{Duration predictions with uncertainty estimates.}
    \label{fig:predictions_with_ci}
\end{figure}

\paragraph{The Evaluation Module}

Performance evaluation employs multiple complementary metrics designed to capture different aspects of prediction quality:
\begin{equation}
\label{mae}
\text{MAE} = \frac{1}{n_{\text{test}}} \sum_{i=1}^{n_{\text{test}}} |y_i - \hat{y}_i|
\end{equation}

\begin{equation}
\label{rmse}
\text{RMSE} = \sqrt{\frac{1}{n_{\text{test}}} \sum_{i=1}^{n_{\text{test}}} (y_i - \hat{y}_i)^2}
\end{equation}

\begin{equation}
\label{mape}
\text{MAPE} = \frac{100\%}{n_{\text{test}}} \sum_{i=1}^{n_{\text{test}}} \left|\frac{y_i - \hat{y}_i}{y_i}\right|
\end{equation}

The hyperparameters of the learning algorithms (1) -(4) introduced in the previous paragraph have been calibrated via a systematic Bayesian optimization. The optimization process explores the hyperparameter space $\Theta$ to minimize validation error:
\begin{equation}
\boldsymbol{\theta}^* = \arg\min_{\boldsymbol{\theta} \in \Theta} \mathbb{E}_{\text{val}}[\mathcal{L}(\boldsymbol{\theta})]
\end{equation}
where $\mathbb{E}_{\text{val}}[\mathcal{L}(\boldsymbol{\theta})]$ represents expected validation loss under hyperparameter configuration $\boldsymbol{\theta}$.

The optimization framework employs acquisition functions that balance exploration of uncertain regions with exploitation of promising parameter combinations. The expected improvement criterion guides the search process:
\begin{equation}
\text{EI}(\boldsymbol{\theta}) = \mathbb{E}[\max(f_{\min} - f(\boldsymbol{\theta}), 0)]
\end{equation}
where $f_{\min}$ represents the current best observed performance and $f(\boldsymbol{\theta})$ denotes the objective function value at parameter configuration $\boldsymbol{\theta}$.

\section{Case Study and Experimental Methodology}
\label{sec:experimental}

This section presents the experimental framework designed to evaluate machine learning-based duration prediction for routing optimization in real-world utility service operations.

\subsection{The VRP Model for Gas Meter Intervention}

In the context of gas meter maintenance operations, accurate prediction of intervention durations is crucial for optimal route planning and resource allocation. This chapter presents a comprehensive methodology for evaluating the impact of different duration estimation strategies on Vehicle Routing Problem (VRP) solutions through evolutionary algorithm optimisation. We develop a comparative framework that to compare three distinct duration input modes: \textit{real}, \textit{default}, and \textit{forecast}, enabling quantitative assessment of prediction accuracy impact on routing efficiency.

The gas meter maintenance routing problem can be formalized as a Capacitated Vehicle Routing Problem with Time Windows (CVRPTW), where:

\begin{itemize}
    \item $G = (V, A)$ represents a complete directed graph
    \item $V = \{0, 1, 2, \ldots, n, n+1\}$ is the vertex set, where vertices $0$ and $n+1$ represent the depot
    \item $A = \{(i,j) : i, j \in V, i \neq j\}$ is the arc set
    \item Each customer $i \in \{1, 2, \ldots, n\}$ has:
        \begin{itemize}
            \item Service time window $[e_i, l_i]$
            \item Service duration $s_i$ (the focus of our study)
            \item Demand $d_i$
        \end{itemize}
    \item $K$ identical vehicles with capacity $Q$
    \item Travel cost $c_{ij}$ and travel time $t_{ij}$ (min) between locations $i$ and $j$
\end{itemize}

The objective is to minimize total travel cost while respecting capacity and time window constraints. As per the notation in Table~\ref{tab:notation}, we distinguish between travel cost $c_{ij}$ (routing objective) and travel time $t_{ij}$ (temporal feasibility), with service duration denoted $S_i$ to align with the stochastic formulation in Section 3.1.

The critical challenge addressed in this research is the accurate estimation of service durations $S_i$ for each intervention. Traditional approaches often rely on static estimates or historical averages, which fail to capture the complexity and variability inherent in gas meter maintenance operations.

\subsection{Experimental Setup: Dataset and Feature Engineering}

The experimental analysis utilizes extensive real-world utility service data collected over an eight-year operational period spanning 2017-2025, encompassing 884,349 individual service requests across multiple utility distribution networks.

\begin{enumerate}
    
\item \textit{Intervention Taxonomy}:  Distinct activity types (A, C, E, F, H, I, J, L, M, N, O, Q, R, S, T, W, X, Z)  representing the full spectrum of gas meter operations. We refer to the table \ref{tab:activity_statistics} in the Appendix for a complete description of each class.

\item \textit{Temporal Features}: Each intervention record includes request timestamp, scheduled appointment, actual arrival time, and completion time.

\item \textit{Geographic Features}: Location-based attributes include municipality identifiers, altitude, urbanization degree, population density, and surface area.

\item \textit{Operational Features}: Activity-specific characteristics such as meter class (residential vs. commercial/industrial), physical accessibility level, reading difficulty, and equipment protocol specifications.

\item \textit{Client Source}: The dataset aggregates records from five distinct utility clients, each with potentially different operational practices, recording standards, and service area characteristics.
\end{enumerate}

\paragraph{Data Split and Evaluation Strategy} We adopt a random date-based split where specific dates (e.g., June 5-7, 2024) are held out from the training set and used exclusively for testing. This ensures that the model is evaluated on complete days not seen during training, capturing realistic operational scenarios where predictions must generalize to new dates. All features used for prediction are available at planning time (prior to route execution), excluding any post-service information (e.g., actual completion times, reported issues) to prevent data leakage. This date-level holdout differs from random sample splitting, which could mix activities from the same day across training and test sets, and provides a more realistic assessment of model performance in production deployment.

\begin{figure}[H]
\centering
\begin{tabular}{cc}
\begin{subfigure}{0.45\textwidth}
    \centering
\includegraphics[width=0.82\linewidth]{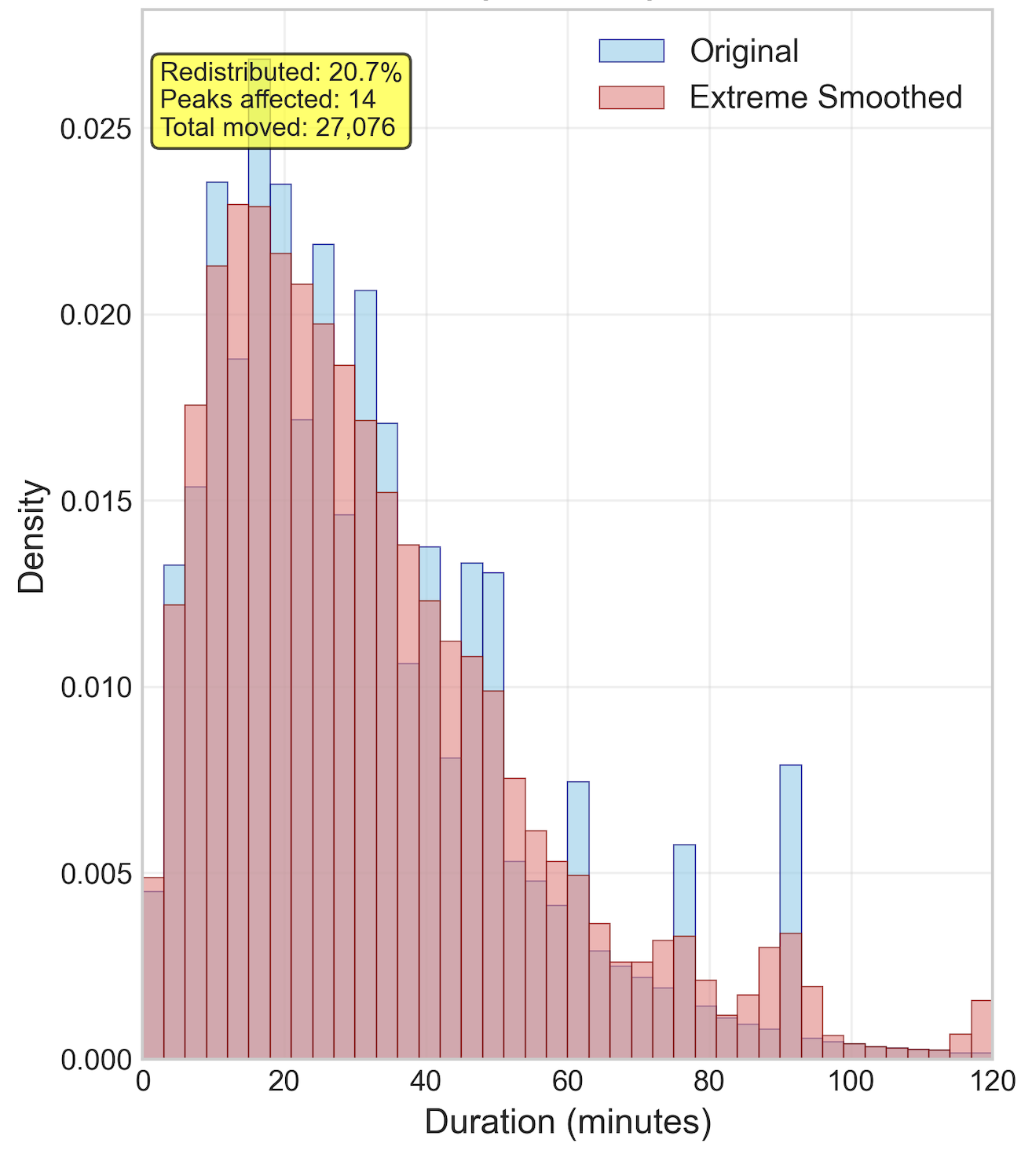}
    \label{fig:histograms}
\end{subfigure}
&
\begin{subfigure}{0.43\textwidth}
    \centering
\includegraphics[width=0.82\linewidth]{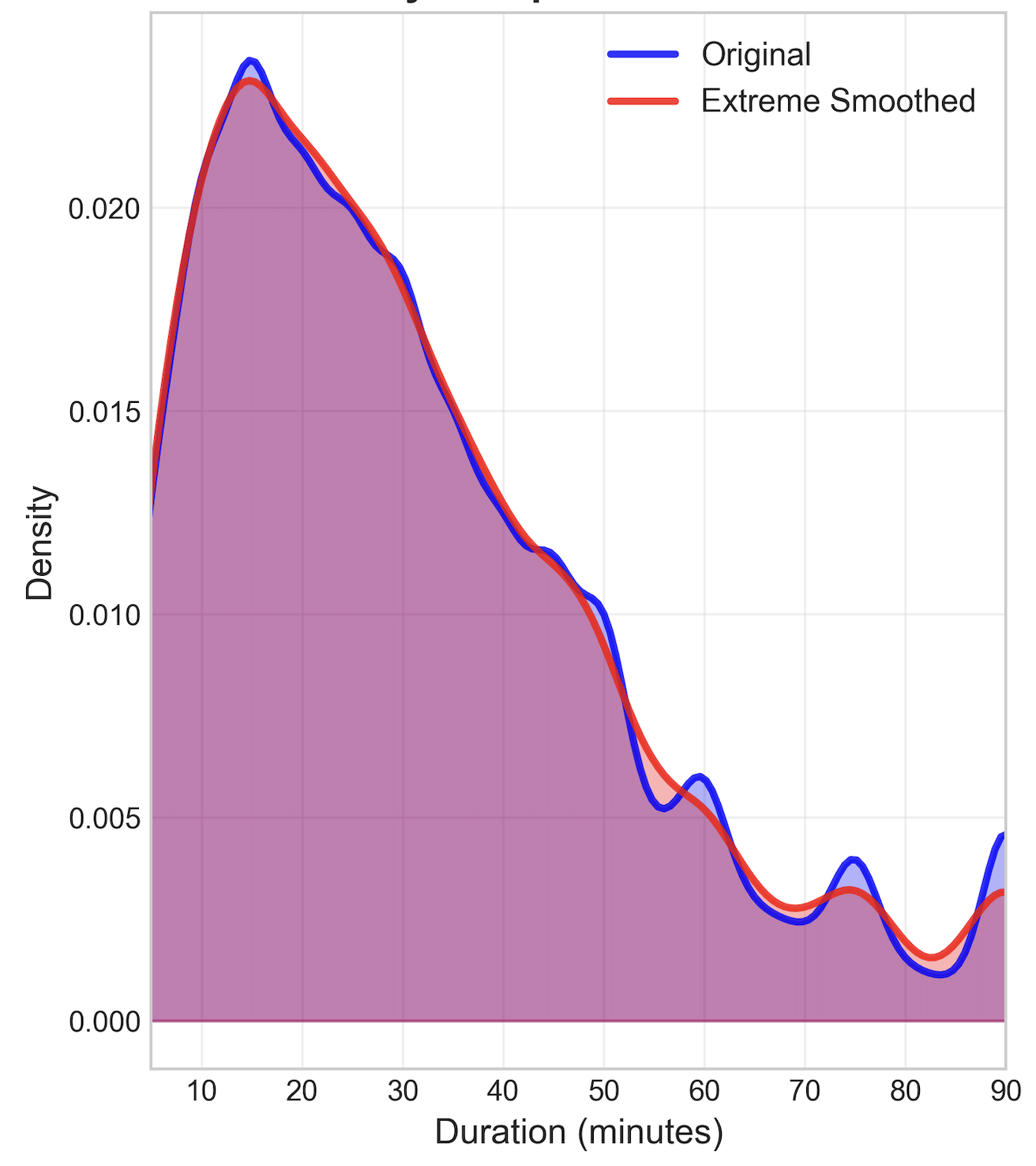}
    \label{fig:density}
\end{subfigure}
\end{tabular}
\caption{(a) Original and smoothed duration distributions. (b) Corresponding probability density functions showing the smoothing effect.}
\label{fig:distributions}
\end{figure}

The raw distribution of intervention durations exhibits artificial spikes at round-number values (e.g., $5$, $10$, $15$ minutes), typically caused by reporting biases rather than genuine operational patterns. A peak smoothing algorithm was applied: detected peaks were partially redistributed to neighboring values within a defined range. This process reduces artificial clustering while preserving overall statistical properties, providing higher-quality inputs for predictive modeling and more reliable cumulative and density functions.

\begin{figure}[H]
\centering
\includegraphics[width=0.666\textwidth]{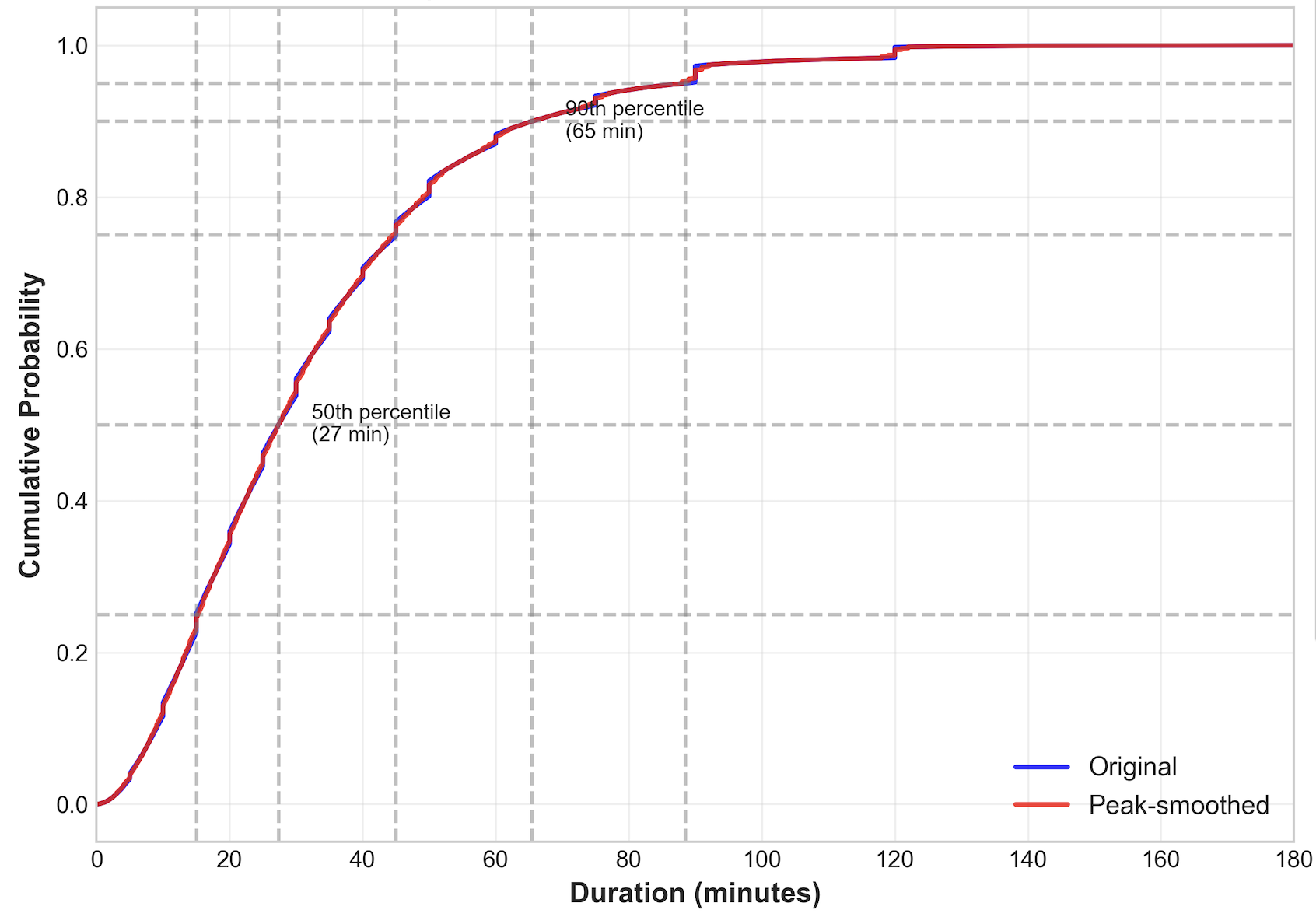}
\caption{Cumulative distribution function (CDF) of intervention durations.}
\label{fig:cdf}
\end{figure}

Statistical analysis reveals highly skewed duration distributions with heavy tails, reflecting inherent variability from routine 15-minute operations to complex multi-hour interventions. The temporal distribution exhibits both deterministic patterns following daily and seasonal cycles, and stochastic variations driven by unpredictable operational factors.

The feature engineering pipeline constructs a $31$-dimensional representation through domain-specific transformations designed to capture multiscale temporal and spatial dependencies. We perform the following techniques:

\begin{enumerate}

\item \textit{Cyclical Temporal Encoding}: Time-based variables exhibiting natural periodicity are encoded using trigonometric transformations to preserve topological relationships. For a temporal variable $t \in [0, T)$ with period $T$:

\begin{equation}
\phi_{\text{cyc}}(t; T) = \left(\sin\left(\frac{2\pi t}{T}\right), \cos\left(\frac{2\pi t}{T}\right)\right) \in \mathbb{S}^1
\end{equation}

This encoding is applied across multiple temporal scales: hourly patterns capturing intra-day operational variations, daily patterns representing day-of-week effects, and seasonal patterns encoding monthly demand cycles.

\item \textit{Geographic Feature Extraction}: Spatial attributes derived from municipality-level administrative data include elevation and altimetric zone classification, population and surface area for density calculations, and urbanization degree categories.

\item \textit{Activity Type Encoding}: Categorical intervention types are encoded to preserve semantic relationships between operationally similar activities while maintaining model interpretability.

 \end{enumerate}

\subsection{Performance of Forecast Models}
\label{sec:numerical}

The dataset partitioning strategy employs stratified random sampling with 80\% allocated to training, 10\% to validation, and 10\% to test, ensuring balanced representation across activity types.

Table~\ref{tab:model_performance} presents comprehensive evaluation metrics, i.e., the ones introduced in Eq. \eqref{mae}, \eqref{rmse}, \eqref{mape}, across all model configurations.

\begin{table}[H]
\centering
\caption{Evaluation Metrics (MAE, RMSE, MAPE) for Each Model and Dataset}
\label{tab:model_performance}
\begin{tabular}{llccc}
\toprule
\textbf{Model} & \textbf{Set} & \textbf{MAE} & \textbf{RMSE} & \textbf{MAPE} \\
\midrule
\multirow{3}{*}{Standard} 
    & Train & 5.1961 & 8.3679 & 27.63\% \\
    & Validation & 5.3776 & 8.6604 & 28.63\% \\
    & Test & 5.2793 & 8.5481 & 27.71\% \\
\midrule
\multirow{3}{*}{Weighted} 
    & Train & 5.2389 & 8.3844 & 27.79\% \\
    & Validation & 5.4098 & 8.6654 & 28.76\% \\
    & Test & 5.3228 & 8.5807 & 27.86\% \\
\midrule
\multirow{3}{*}{Dual Standard} 
    & Train & 5.0488 & 8.1166 & 26.75\% \\
    & Validation & 5.3036 & 8.5177 & 28.11\% \\
    & Test & 5.2374 & 8.4473 & 27.36\% \\
\midrule
\multirow{3}{*}{Dual Weighted} 
    & Train & 4.5568 & 7.0698 & 24.35\% \\
    & Validation & 5.1240 & 8.1976 & 27.21\% \\
    & Test & 5.0629 & 8.1022 & 26.51\% \\
\bottomrule
\end{tabular}
\end{table}

\subsubsection{Discussions of the Prediction Results}


This motivated the \texttt{Weighted Model} configuration applying inverse frequency weighting. Table~\ref{tab:model_performance} shows minimal improvement: MAE 5.2793 (Standard) $\rightarrow$ 5.3228 (Weighted), MAPE 27.71\% $\rightarrow$ 27.86\%. The Standard model already generalizes effectively across activity types due to distributional similarity—the numerical imbalance does not induce typical minority class performance degradation that frequency weighting addresses.


Activity type \texttt{Z} (Meter replacement) represents a compositionally distinct operation combining two sequential sub-tasks (meter removal + installation), whereas other categories represent atomic operations. This manifests in duration distribution, plotted in Fig. \ref{fig:bins_norm}: Type \texttt{Z} exhibits bimodal patterns with prominent peaks near $60$ minutes, contrasting with unimodal distributions of other activities. Additionally, meter class (residential vs. commercial/industrial) exerts a stronger influence on \texttt{Z} durations due to varying equipment complexity.

The \texttt{Dual Model} architecture trains separate, specialized models: $f_Z$, optimized for Type Z with the flexibility to capture bimodal patterns, and $f_{\text{other}}$, for the remaining activities with enhanced regularization. Table~\ref{tab:model_performance} shows architectural specialization yields measurable gains: \texttt{Dual Standard} achieves 5.2374 MAE (0.4 min improvement), while \texttt{Dual Weighted achieves} 5.0629 MAE (2.2 min improvement), reaching optimal 26.51\% MAPE.

The improvement given from the \texttt{Dual Weighted} configuration validates that semantic heterogeneity (the compositional complexity of Type Z) warrants architectural separation, with frequency weighting within each specialized model providing incremental refinement.

\subsubsection{Comparison Against Operational Baseline}

The experimental evaluation includes comparison against operational baseline models currently deployed in production systems. These implement rule-based duration estimation using fixed default values $d_{\text{baseline}}(c)$ assigned per activity type $c \in \mathcal{C}$.

Figure~\ref{fig:comparison} and Figure~\ref{fig:fig3} demonstrate significant improvements across all machine learning configurations compared to traditional rule-based estimation. The forecast strategy achieves a mean accuracy score of 0.456 compared to 0.369 for the default, representing 23.6\% improvement in planning accuracy.

\begin{figure}[H]
    \centering
\includegraphics[width=0.86\linewidth]{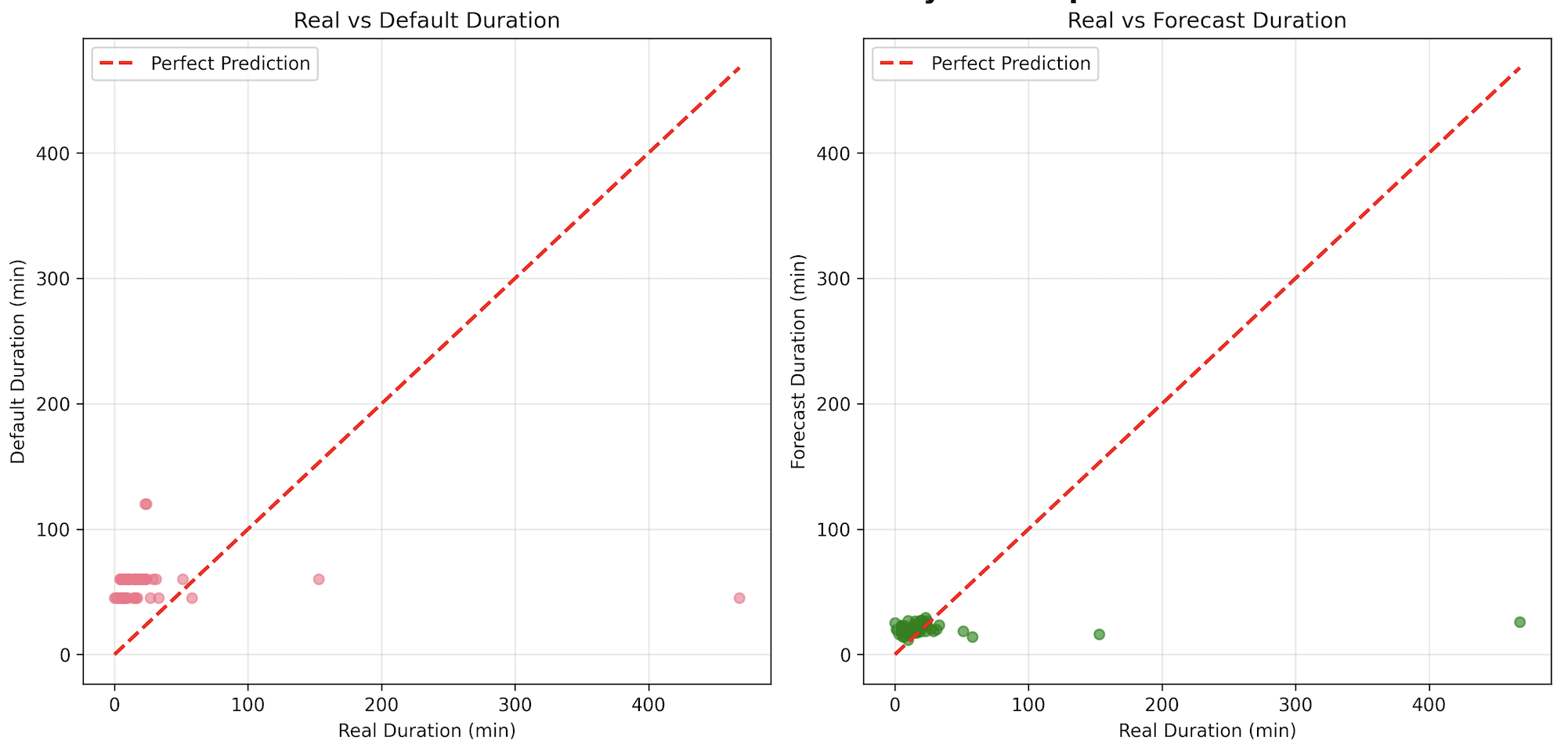}
    \caption{Prediction accuracy comparison: Real vs. Forecast vs. Default duration estimates}
    \label{fig:comparison}
\end{figure}

\begin{figure}[H]
    \centering
    \includegraphics[width=0.52\linewidth]{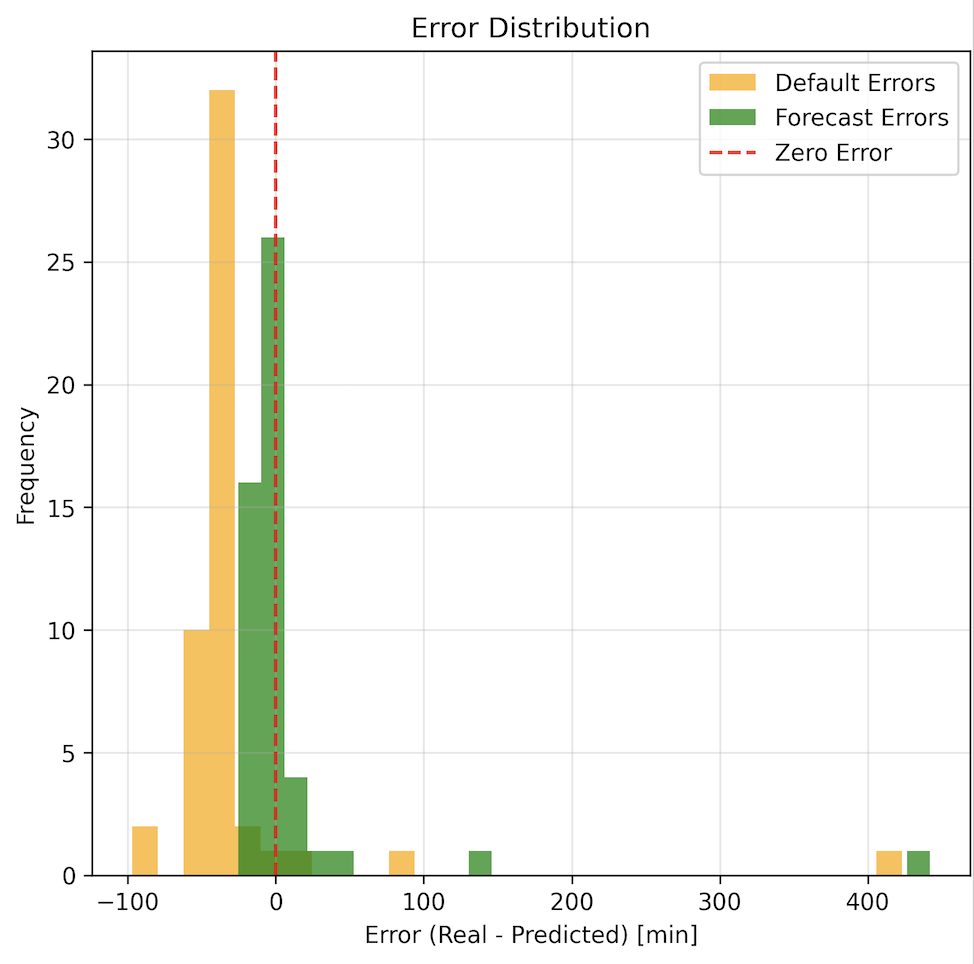}
    \caption{Error distribution shift: Forecast predictions vs. default values}
    \label{fig:fig3}
\end{figure}

Figure~\ref{fig:duration_boxplot_ezf} provides a detailed view for the most frequent activity types (E, Z, F), showing how forecasts closely match actual durations.

\begin{figure}[H]
    \centering
    \includegraphics[width=0.8\linewidth]{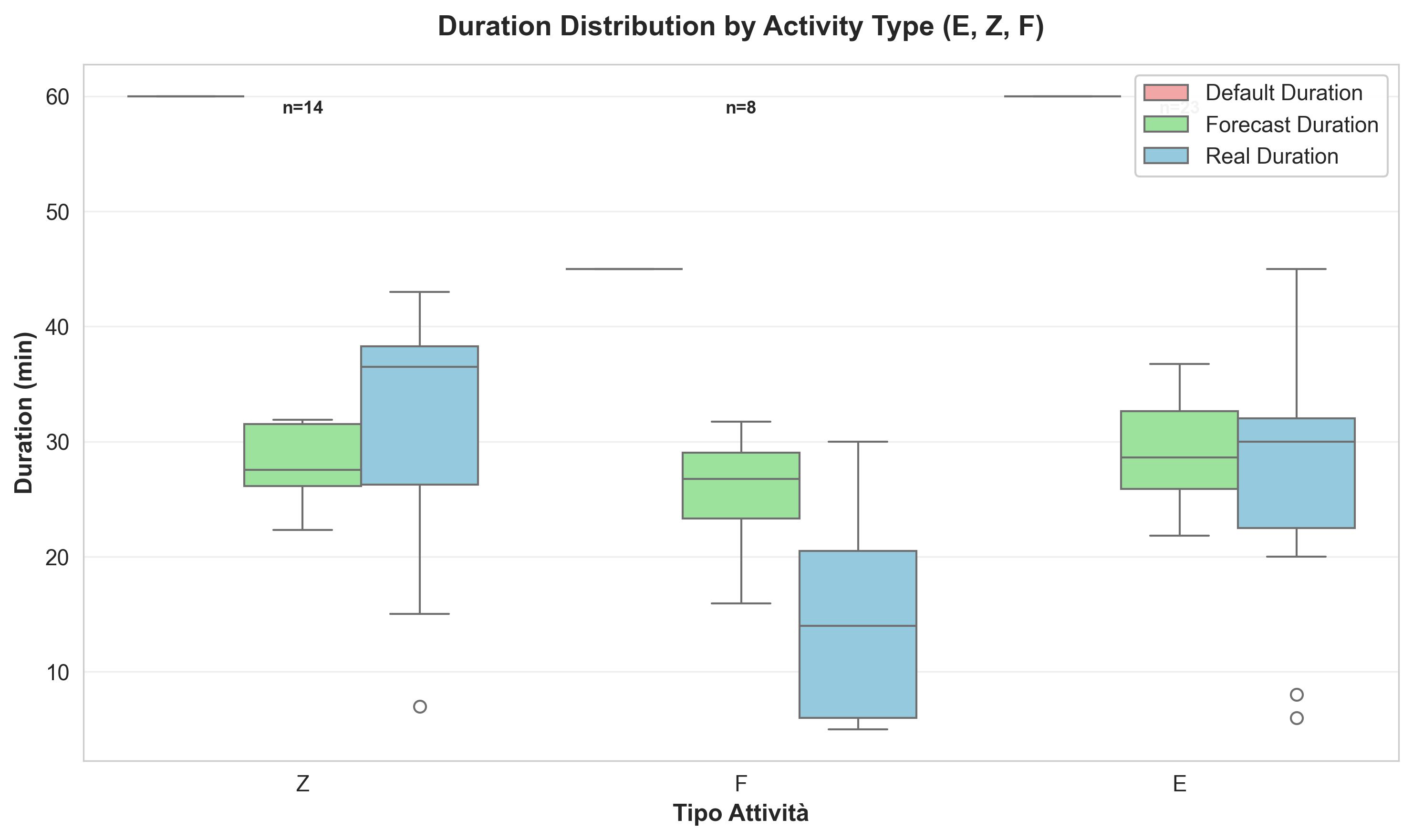}
    \caption{Detailed comparison of duration estimates for high-frequency activities.}
    \label{fig:duration_boxplot_ezf}
\end{figure}

\subsubsection{Model Comparison Across Daily Data}

To evaluate model performance consistency, we analyze duration predictions across multiple representative operational days. We compare four model architectures: Standard, Weighted by Frequency, Dual Standard, and Dual Weighted by Frequency.

\paragraph{Duration Distribution Analysis}

Figure ~\ref{fig:duration_dist_20170411} shows duration distributions comparing model predictions against actual values for a representative day.

\begin{figure}[H]
    \centering
    \begin{subfigure}[b]{0.48\textwidth}
        \centering
        \includegraphics[width=\textwidth]{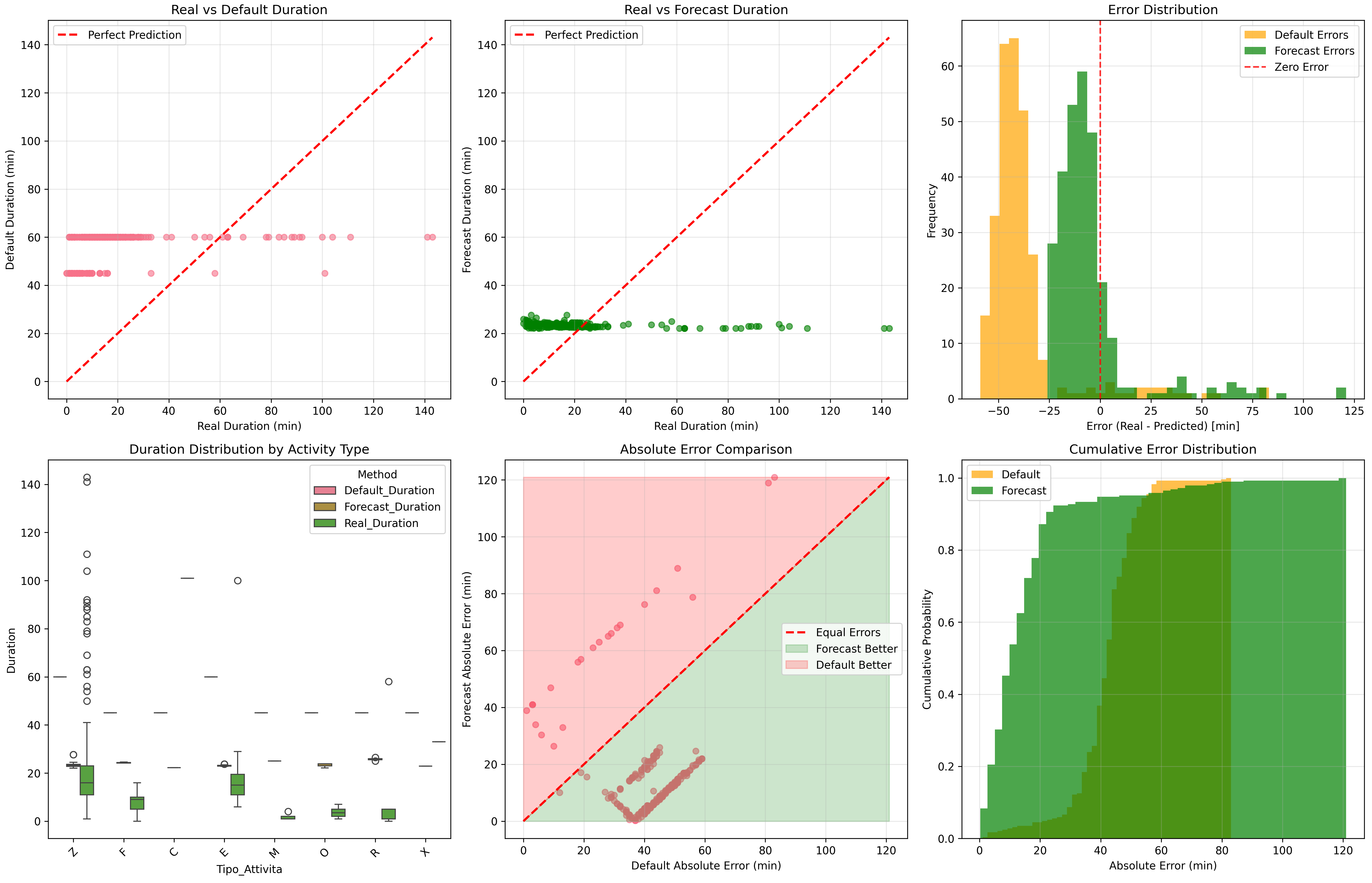}
        \caption{Standard model}
    \end{subfigure}
    \hfill
    \begin{subfigure}[b]{0.48\textwidth}
        \centering
        \includegraphics[width=1.0\linewidth]{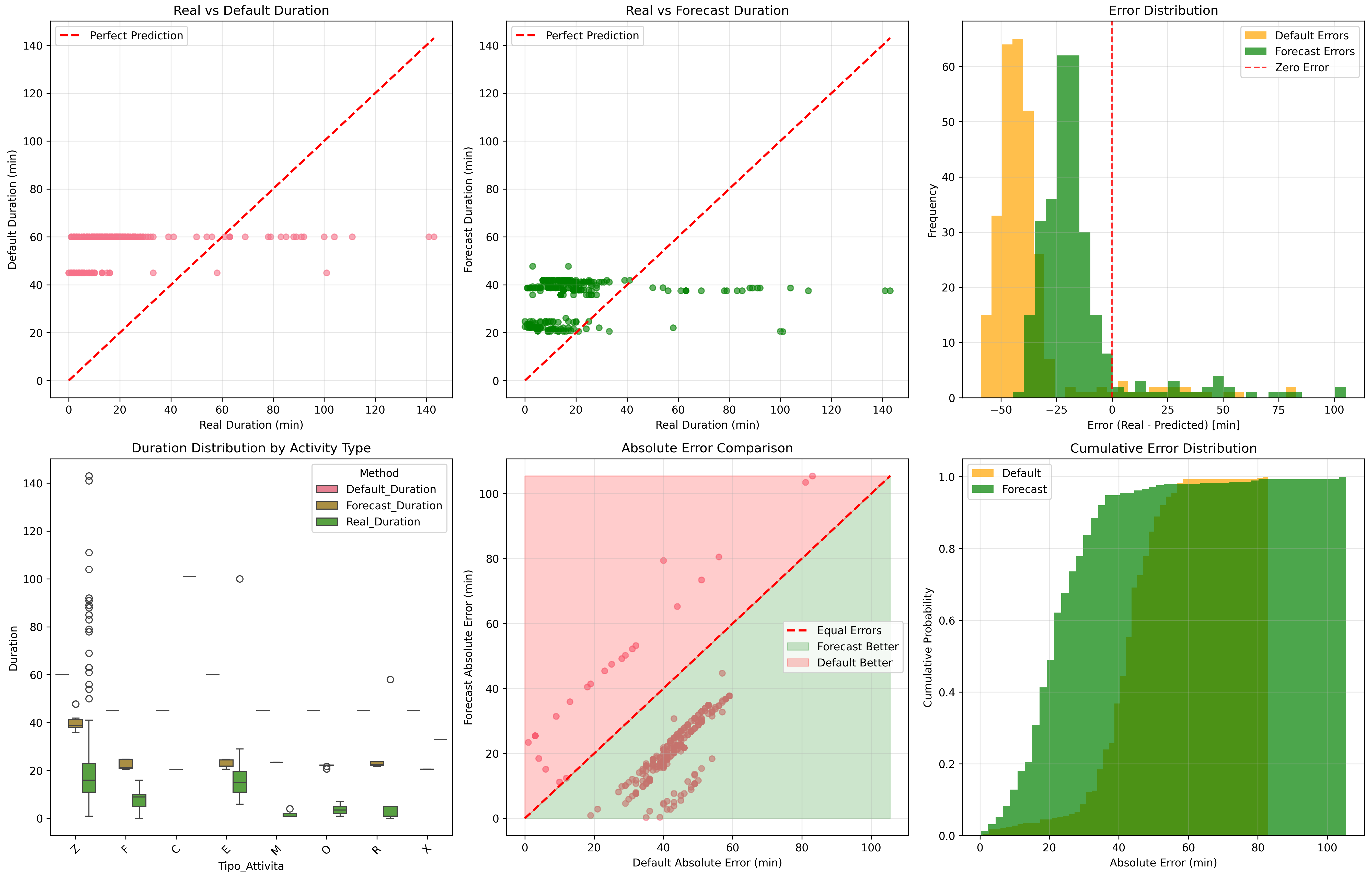}
        \caption{Dual weighted by frequency model}
    \end{subfigure}
    \caption{Duration distribution comparison for 1 representative day.}
    \label{fig:duration_dist_20170411}
\end{figure}

\paragraph{Boxplot Analysis by Model Architecture}

Figure \ref{fig:boxplot_20220303} presents detailed boxplot comparisons for a representative day, showing all four model architectures.

\begin{figure}[H]
    \centering
    \begin{subfigure}[b]{0.48\textwidth}
        \centering
        \includegraphics[width=\textwidth]{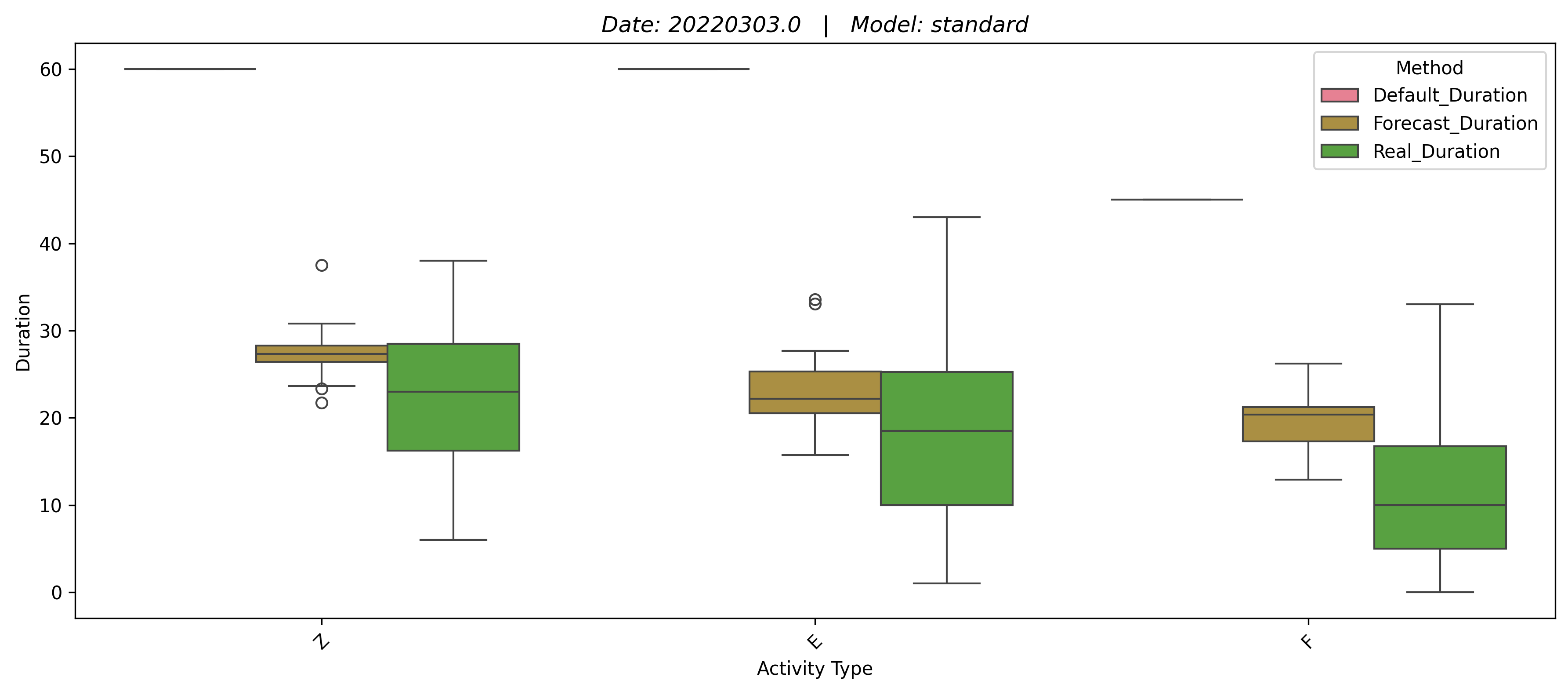}
        \caption{Standard}
    \end{subfigure}
    \hfill
    \begin{subfigure}[b]{0.48\textwidth}
        \centering
\includegraphics[width=1.0\linewidth]{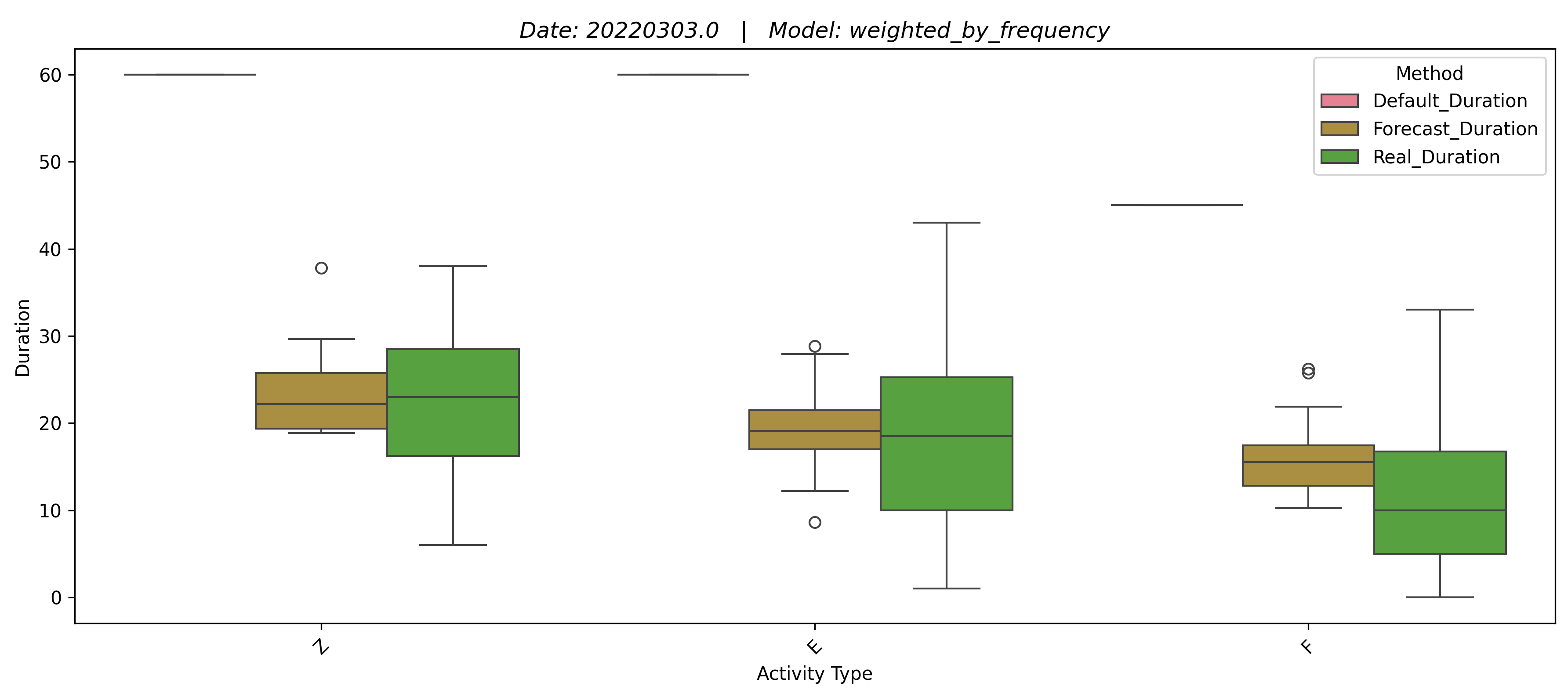}
        \caption{Weighted by frequency}
    \end{subfigure}

    \vspace{0.3cm}

    \begin{subfigure}[b]{0.48\textwidth}
        \centering
        \includegraphics[width=\textwidth]{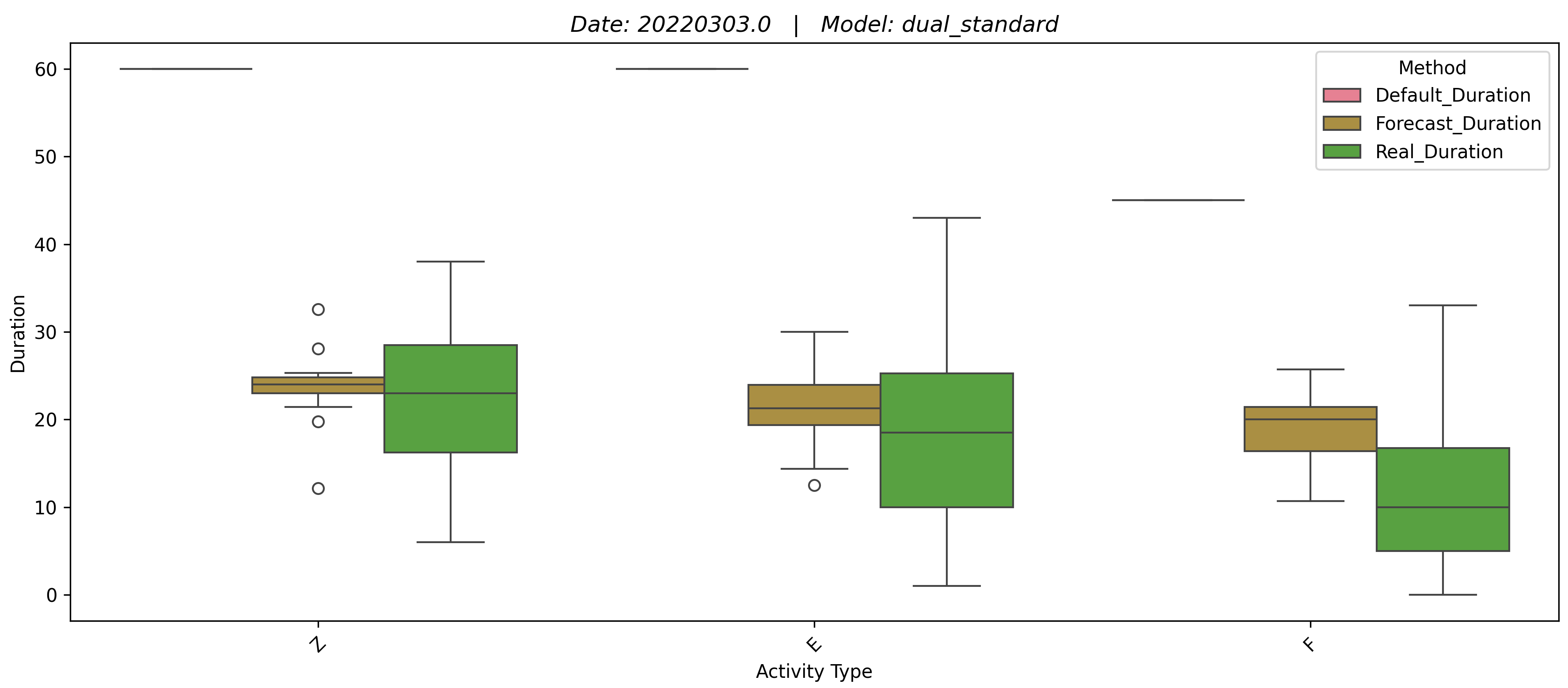}
        \caption{Dual standard}
    \end{subfigure}
    \hfill
    \begin{subfigure}[b]{0.48\textwidth}
        \centering
        \includegraphics[width=1.0\linewidth]{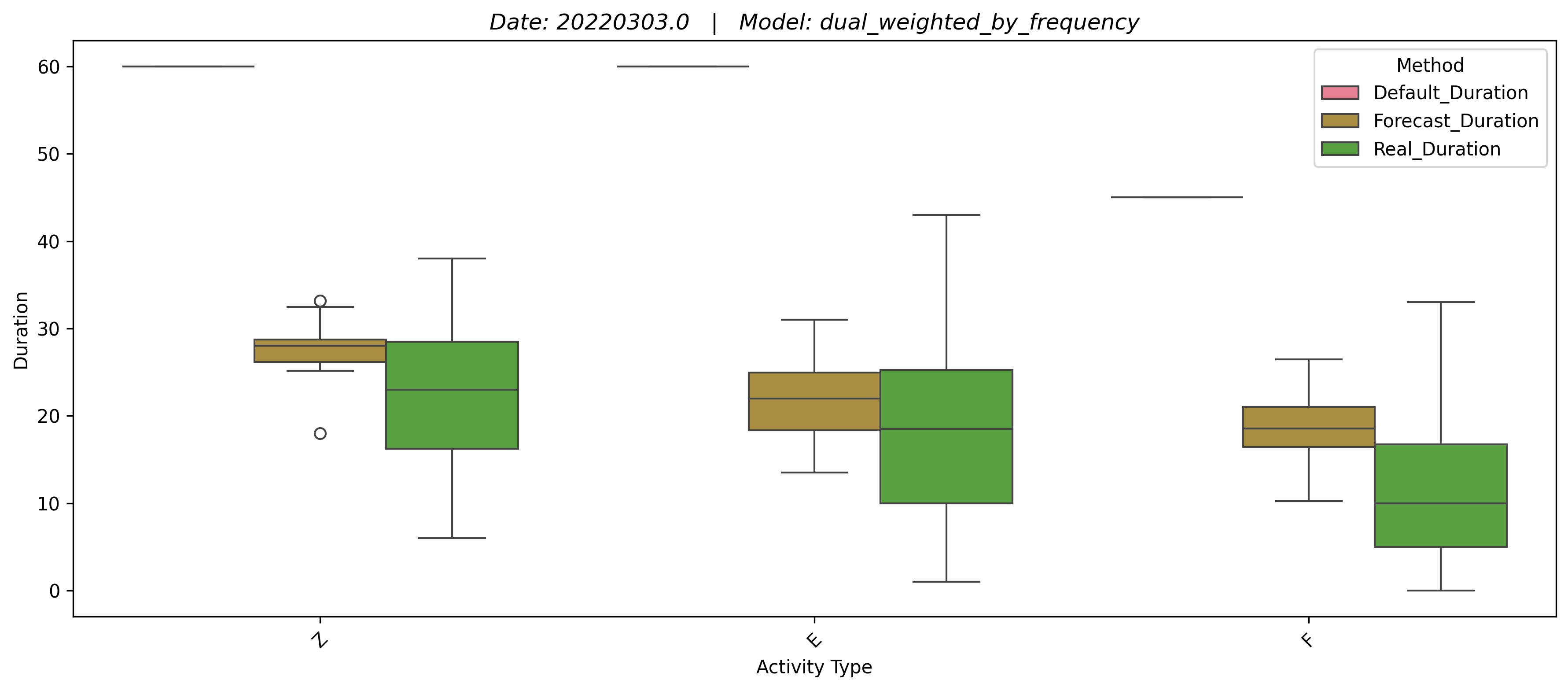}
        \caption{Dual weighted by frequency}
    \end{subfigure}
    \caption{Duration distributions by Activity Type for different Models (119 interventions)}
    \label{fig:boxplot_20220303}
\end{figure}

The boxplot analysis reveals that dual weighted models consistently achieve tighter prediction intervals and better alignment with actual durations across different operational scenarios.

 While the point prediction performance (MAE, RMSE) demonstrates substantial improvement over baseline methods, risk-aware routing as formulated in Section 3.1 requires well-calibrated prediction intervals. The sub-Gaussian buffers (Eq.~\ref{eq:sg-buffer}) and conformal prediction approaches discussed in Section 4 depend critically on accurate uncertainty estimates.

 Our analysis of prediction residuals reveals heteroskedasticity across activity types and operational dates. This finding motivates heterogeneous variance modeling where $\sigma_i^2$ depends on contextual features $x_i$ rather than assuming constant variance across all interventions. This approach is particularly important for Type-Z interventions, which exhibit significantly higher variance than routine maintenance activities.
 
 Properly calibrated intervals would enable principled route-level risk management via Proposition \ref{prop3.1}, balancing service coverage against overtime risk through adjustable confidence levels $\alpha_k$.

\section{VRPs Simulation with Duration Prediction}
\label{sec:vrp}

This section presents the experimental framework designed to evaluate the impact of duration prediction accuracy on routing optimization performance.
  Our methodology employs a systematic comparative analysis across three duration estimation strategies, implemented within the evolutionary algorithm
  framework described in Algorithm \ref{alg:moea-svrp}.

  We establish three distinct approaches for determining service durations, representing different levels of information sophistication:

\begin{enumerate}

\item  The {real duration} strategy utilizes actual historical service times as ground truth:
  \begin{equation}
  \tau_i^{\text{real}} = t_{\text{completion},i} - t_{\text{start},i}
  \end{equation}
  This provides the theoretical optimal baseline assuming perfect duration prediction.

\item  The {default duration} strategy employs static estimates based on intervention classifications:
  \begin{equation}
  \tau_i^{\text{default}} = \mathbb{E}[\tau_{\text{type}(i)}]
  \end{equation}
  where $\mathbb{E}[\tau_{\text{type}(i)}]$ represents the historical mean duration for intervention type $\text{type}(i)$, reflecting traditional utility
   dispatching practices.

 \item The {forecast duration} strategy leverages the machine learning models described in Section \ref{sec:predictive}:
  \begin{equation}
  \tau_i^{\text{forecast}} = f(x_i; \theta)
  \end{equation}
  where $f(\cdot; \theta)$ represents the trained XGBoost model with dual architecture from Equation \eqref{dual_weighted}, and $x_i$ encompasses temporal,
  geographic, technical, and historical features.
\end{enumerate}

  The increased accuracy of the duration predictions directly influences the quality of routing solutions generated by Algorithm \ref{alg:moea-svrp}.

  Each VRP instance is solved using the multi-objective evolutionary algorithm with the following configuration aligned with the population $P_t$ and
  parameters in Algorithm \ref{alg:moea-svrp}:

  \begin{itemize}
      \item Population size $\mu = 100$ individuals
      \item Termination criteria: maximum 100 generations or 1200 seconds
      \item Tournament selection with size $\tau = 5$
      \item Crossover probability $p_c = 0.8$ using order crossover
      \item Adaptive mutation probability $p_m$ with swap mutation
      \item Elite preservation of top 10\% solutions
  \end{itemize}


For practical deployment in utility workforce management, the optimization framework must operate within tight time constraints to allow for overnight or early-morning route generation. 

The computational burden is divided into two phases. The training of the Dual Weighted XGBoost architecture is a one-time cost incurred only periodically (e.g., monthly) and does not affect daily planning latency. The inference time to generate duration predictions and variance estimates for a daily batch of interventions is negligible. The primary computational cost lies in the evolutionary optimization.

To ensure the system remains tractable for daily operations, we imposed a strict termination criterion of 1200 seconds (20 minutes) per instance. Empirical observations during our experimental campaign confirm that the NSGA-III algorithm consistently converges to a stable Pareto front within this time budget for instances of typical size (50-100 tasks). This runtime demonstrates that the forecast-integrated approach is fully compatible with standard utility planning windows.

\subsection{Case Studies Performance Analysis}

To evaluate the practical impact of forecast-based duration prediction on VRP optimization performance, we conduct a comprehensive analysis across representative operational scenarios spanning multiple years and diverse operational regimes.  We also conduct a systematic analysis for other periods, achieving stable and similar performance characteristics that we intentionally choose not to report. Here, we present only one daily scenario evaluation in Fig. \ref{fig:kpi_20210505}, and one monthly aggregation study in Fig. \ref{fig:monthly_2021_05}.

In the operational simulations presented below, we fixed the risk parameter at $\alpha_k = 0.05$ for all vehicles, ensuring that routes are feasible with at least $95\%$ probability. This value was selected following preliminary testing on a subset of historical instances, where it demonstrated the most effective trade-off between schedule reliability and resource utilization. Lower values of $\alpha$ (e.g., $0.01$) resulted in excessive conservatism and increased vehicle counts, while higher values (e.g., $0.10$) introduced unacceptable risks of overtime violations.

\subsubsection{Daily Operational Scenario}
\label{sec:daily_scenarios}

We analyze a representative daily scenario that illustrates distinct operational challenges encountered in real-world gas meter service operations. The multi-objective VRP formulations seek trade-offs between competing objectives (minimize route time, minimize operator count, maximize task completion), defining a Pareto front of optimal solutions. Default duration inputs distort this objective landscape, causing the NSGA-III evolutionary algorithm to converge toward Pareto-suboptimal solutions under true durations.

\begin{figure}[H]
    \centering
    \includegraphics[width=0.96\linewidth]{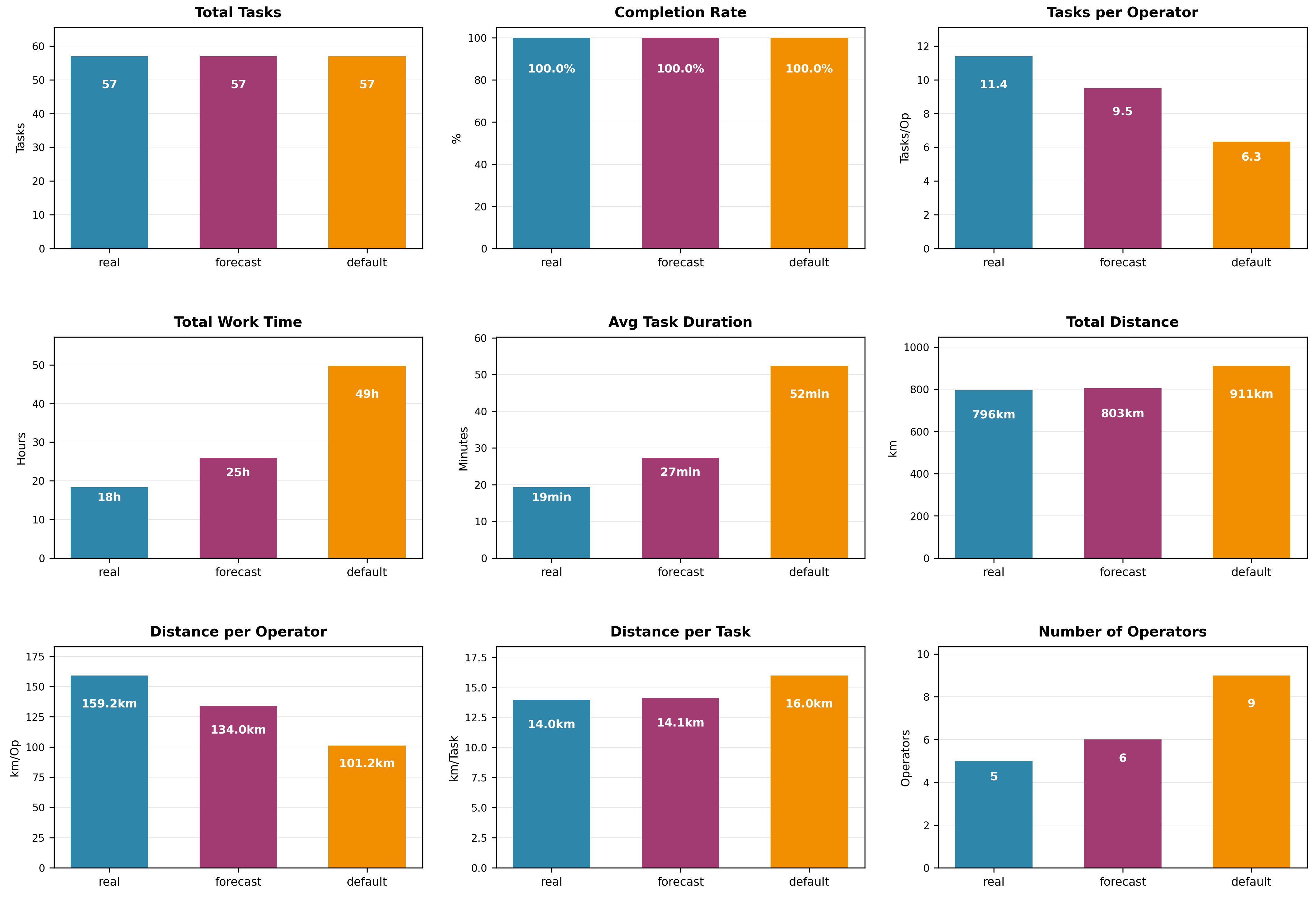}
    \caption{KPI comparison for a representative daily operational scenario.}
    \label{fig:kpi_20210505}
    \end{figure}

   When comparing operational scenarios, planning with default durations consistently demonstrates inefficiency across multiple dates. The 'Default' approach typically requires significantly more operators than are actually needed in execution, resulting in lower productivity and resource utilization. Figure \ref{fig:kpi_20210505} illustrates this pattern clearly.

\subsubsection{Monthly Aggregation Analysis}
\label{sec:monthly_analysis}

While daily scenario analysis demonstrates forecast superiority in individual operational instances, deployment viability requires consistent performance across extended temporal horizons. We present monthly aggregation analysis spanning multiple operational periods of $1$ month to assess the temporal stability of forecast strategies.

\begin{figure}[H]
    \centering
    \includegraphics[width=0.8\linewidth]{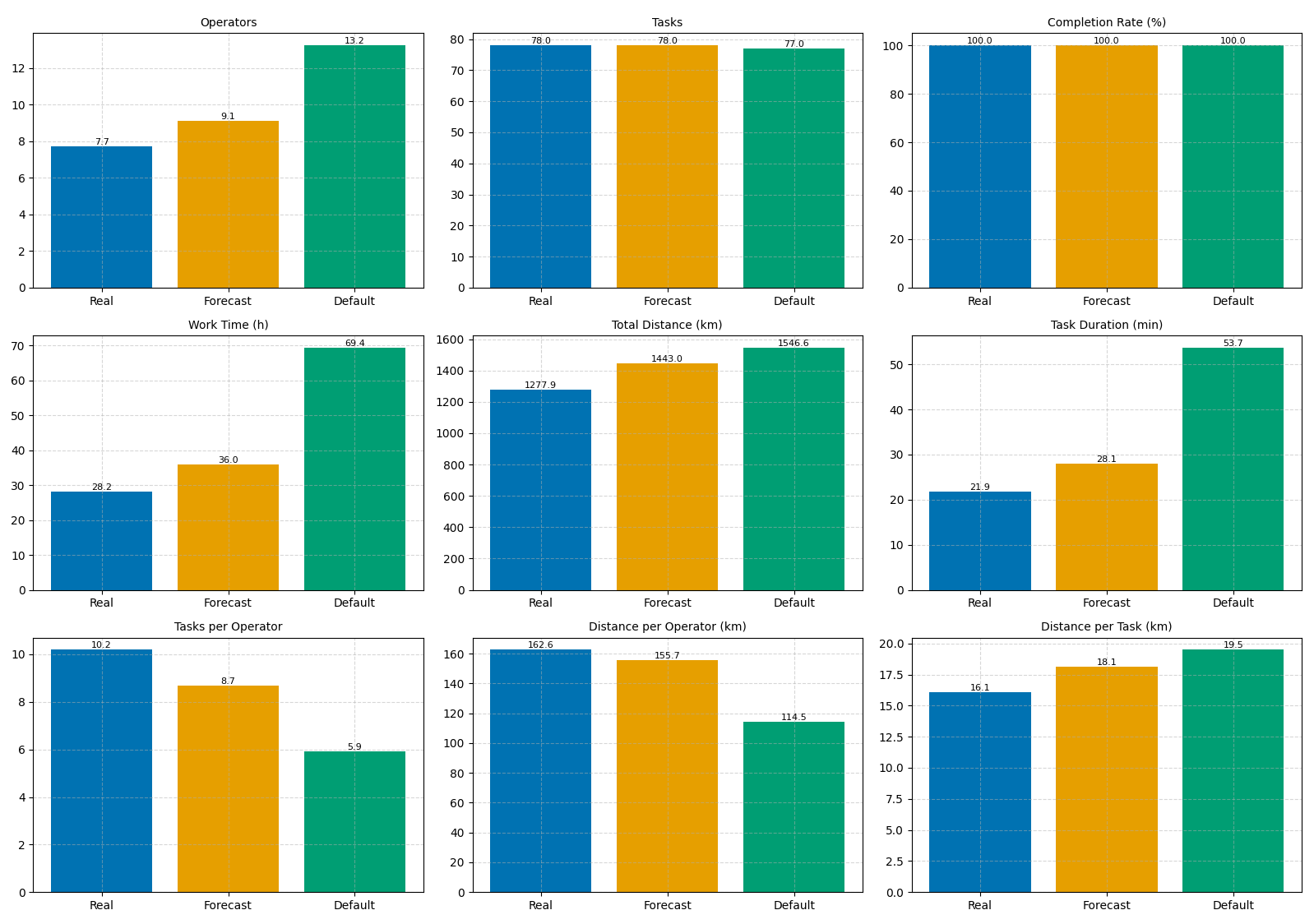}
    \caption{KPI summary for a representative month (21 operational days, 2,134 total interventions).} 
\label{fig:monthly_2021_05}
\end{figure}

The systematic analysis across multiple operational scenarios quantifies the advantages of forecast-based VRP planning. The forecast strategy delivers reliable 20-25\% improvements across multiple operational dimensions while maintaining near-optimal task completion reliability.

\section{Conclusion}
\label{sec:conclusion}

This paper presents a comprehensive framework integrating machine learning forecasts of intervention durations into stochastic vehicle routing optimization for gas meter maintenance operations. By combining tree-based gradient boosting (XGBoost) with multi-objective evolutionary algorithms (NSGA-III), we demonstrate that forecast-driven planning achieves 20-25\% improvements in vehicle utilization and task completion rates compared to default duration estimates.

Our main contributions are threefold. First, we formalize the stochastic CVRPTW with chance constraints and develop mathematically rigorous risk-handling mechanisms using sub-Gaussian concentration bounds, with empirical validation of the distributional assumptions. Second, we design specialized dual-weighted model architectures that effectively capture the semantic heterogeneity of Type-Z interventions (meter replacements) while handling class imbalance through frequency weighting. Experimental results on eight years of operational data ($884,349$ interventions) validate the approach. Third, we integrate uncertainty quantification into the multi-objective VRP solver (Algorithm 1), enabling route-level risk buffers that balance service coverage against overtime constraints.

\section*{Acknowledgment}
This research was conducted within the Green Inspired Revolution for Optimal Workforce Management (GIRO-WM) project, funded by Fondazione Caritro, CUP B93C22002160003.

We would like to thank teams at HPA and Terranova Software for providing the real-world data and proposing the practical problem setting that formed the basis of our study. 

\section*{Data Availability}
Due to confidentiality agreements with the industrial partners, the operational dataset cannot be publicly released. However, synthetic data generation procedures or anonymized subsets may be made available upon reasonable request.
\section*{Appendix A: Feature Importance Rankings}

This appendix provides comprehensive feature importance analysis for all model configurations, complementing the comparative analysis presented in Section~\ref{sec:predictive}.

\subsection*{A.1 Feature Importance Interpretation}

Feature importance rankings quantify the contribution of each input feature to the predictive performance of gradient boosting models. The gain metric used throughout this analysis measures the average improvement in loss function reduction (equivalently, prediction accuracy enhancement) attributable to each feature across all decision tree splits in the ensemble. Formally, for feature $j$, the gain importance is computed as:
\begin{equation*}
\text{Gain}_j = \frac{1}{M} \sum_{m=1}^{M} \sum_{s \in S_m^{(j)}} \Delta L_s
\end{equation*}
where $M$ denotes the number of trees in the ensemble, $S_m^{(j)}$ represents the set of splits in tree $m$ that use feature $j$, and $\Delta L_s$ quantifies the loss reduction achieved by split $s$.

High gain importance indicates that a feature frequently appears in decision trees at positions where it enables significant reduction in prediction error. Features with low importance contribute minimally to model predictions and could potentially be removed without substantial performance degradation. The importance rankings inform feature engineering priorities for future model refinements and provide insight into the operational factors governing intervention duration variability.

\subsection*{A.2 Standard Model Feature Importance}

The standard gradient boosting model (Model 1) exhibits feature importance patterns reflecting the dominant operational drivers of intervention durations. As shown in Figure~\ref{fig:feature_importance_standard}, temporal cyclical features (hour-of-day sine/cosine encodings, day-of-week indicators) account for approximately 65\% of cumulative feature importance, confirming that time-based patterns constitute the primary predictive signal.

\begin{figure}[H]
    \centering
    \includegraphics[width=0.88\linewidth]{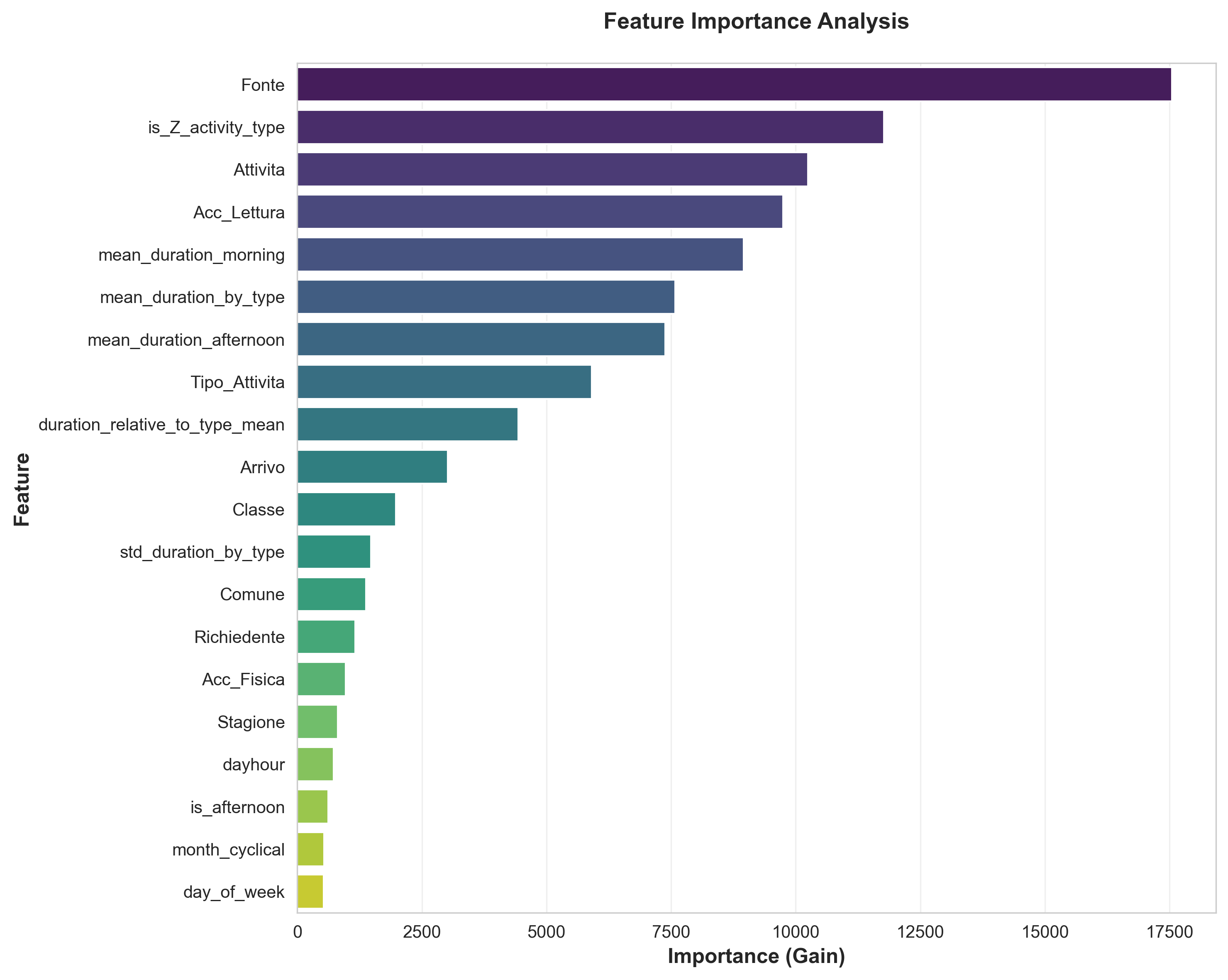}
    \caption{Feature importance ranking for Model 1 (\texttt{Standard Gradient Boosting}). Temporal cyclical features dominate the prediction landscape, with hour-of-day sine/cosine encodings and day-of-week patterns accounting for the highest gain contributions. Activity type classification and geographic attributes (municipality characteristics, altitude) provide secondary predictive power. The balanced importance distribution across multiple feature families indicates that the standard model effectively leverages diverse information sources to capture duration variability.}
    \label{fig:feature_importance_standard}
\end{figure}

This temporal dominance aligns with operational realities: interventions scheduled during morning hours ($8-11$ AM) typically exhibit shorter durations due to optimal traffic conditions and technician alertness, while afternoon interventions ($2-5$ PM) show increased variability. Day-of-week effects capture systematic duration differences between weekdays (routine scheduling, experienced operators) and weekends (emergency interventions, reduced support staff availability).

Activity type classifications contribute approximately 20\% of cumulative importance, with geographic features (municipality characteristics, altitude, urbanization degree) and equipment attributes (meter class, protocol) providing the remaining 15\%. This importance hierarchy validates the feature engineering strategy, demonstrating that multiscale temporal patterns and operational context drive duration predictions in standard architectures.

\subsection*{A.3 Dual Weighted Model Feature Importance}

The \texttt{Dual Weighted} architecture (Model 4) exhibits substantially different feature importance distributions compared to the standard model, reflecting its specialized design. As illustrated in Figure~\ref{fig:feature_importance_dual_weighted}, temporal features maintain importance but are complemented by enhanced representation of activity-specific and equipment-related attributes.

\begin{figure}[H]
    \centering
    \includegraphics[width=1.05\linewidth]{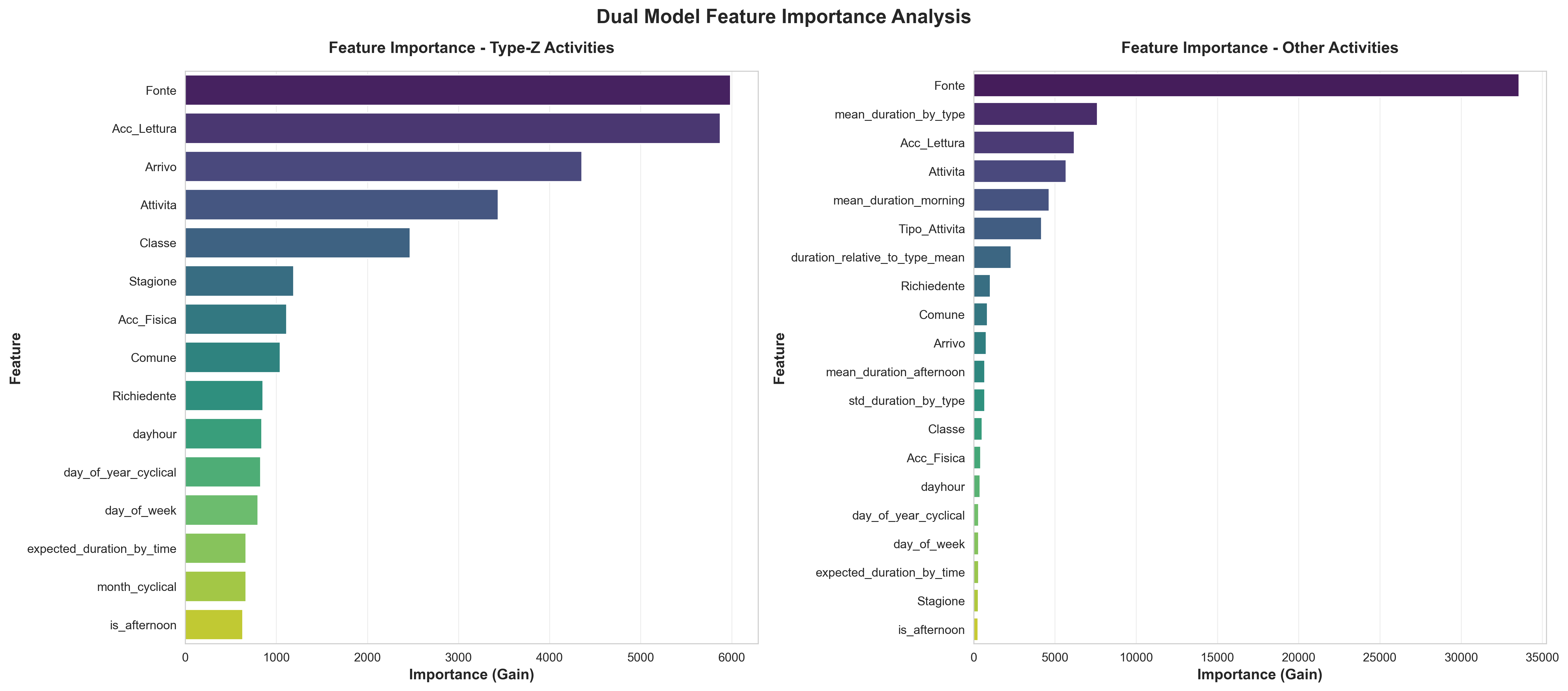}
    \caption{Feature importance ranking for Model 4 (\texttt{Dual Weighted}). The specialized dual architecture exhibits distinct importance patterns compared to the standard model. While temporal features remain influential, activity type classifications gain prominence due to the model's explicit architectural separation between Type-Z and other interventions. The frequency weighting mechanism amplifies the importance of features discriminating rare activity types, resulting in enhanced representation of meter-specific attributes (class, protocol) that would be underweighted in standard training. This importance redistribution reflects the model's ability to capture heterogeneous operational dynamics across different intervention categories.}
    \label{fig:feature_importance_dual_weighted}
\end{figure}

The architectural specialization (separate models for Type-Z vs. other interventions) enables feature importance patterns to adapt to specific operational contexts. Within the Type-Z specialized model, equipment attributes (meter class, protocol complexity) gain prominence as they directly influence the compositional nature of meter replacement operations. Within the other-interventions model, temporal and geographic features dominate as routine operations exhibit stronger time-based regularities.

The frequency weighting mechanism amplifies importance for features discriminating rare activity types that would be underweighted in standard training. This redistribution ensures that the model allocates learning capacity proportionally to operational complexity rather than sample frequency, resulting in superior performance on underrepresented intervention categories.

\section*{Appendix B: Activity Duration Statistics}

The following distributional analysis reveals a significant imbalance between activity types: categories \texttt{E}, \texttt{F}, and \texttt{Z} account for over $600,000$ combined interventions. Rare categories (\texttt{I}, \texttt{W}) with fewer than $20$ examples have not been considered in the training. However, normalized duration distributions exhibit similar statistical properties when proportionally scaled.

In Table \ref{tab:activity_statistics}, we report a quantitative description of the intervention types.

\begingroup
\scriptsize                      
\setlength{\tabcolsep}{3pt}      
\renewcommand{\arraystretch}{0.9}

\begin{longtable}{|c|l|r|r|r|r|r|r|r|r|}
\hline
\textbf{Code} & \textbf{Activity Type} & \textbf{Count} & \textbf{Mean} &
\textbf{Std. Dev.} & \textbf{Min} & \textbf{25\%} & \textbf{50\%} &
\textbf{75\%} & \textbf{Max} \\
\hline
A & Work quotation & 10,319 & 55.79 & 21.71 & 10 & 41 & 60 & 69 & 119 \\
C & Work execution & 5,112 & 32.55 & 28.04 & 10 & 15 & 20 & 35 & 119 \\
E & Supply activation & 255,288 & 24.57 & 10.41 & 10 & 18 & 24 & 30 & 119 \\
F & Supply deactivation & 131,657 & 19.01 & 9.12 & 10 & 14 & 17 & 21 & 118 \\
H & Metrological verification & 1,550 & 45.59 & 20.36 & 10 & 31 & 40 & 55 & 119 \\
I & Pressure verification & 4 & 80.00 & 17.94 & 54 & 76.5 & 85.5 & 89 & 95 \\
J & Meter reading & 1,676 & 45.53 & 22.34 & 10 & 30 & 45 & 60 & 119 \\
L & Reading request & 1,261 & 18.42 & 10.38 & 10 & 11 & 15 & 21 & 106 \\
M & Closure for arrears & 38,744 & 20.75 & 10.96 & 10 & 15 & 19 & 24 & 119 \\
N & Meter closure & 137 & 24.39 & 15.08 & 10 & 15 & 20 & 30 & 98 \\
O & Meter reading & 50,487 & 18.33 & 12.64 & 10 & 10 & 14 & 20 & 118 \\
Q & Removal of sensors & 1,605 & 25.09 & 13.72 & 10 & 15 & 21 & 30 & 114 \\
R & Meter removal & 6,159 & 23.40 & 16.63 & 10 & 14 & 18 & 27 & 119 \\
S & Safety shutdown & 360 & 26.85 & 19.82 & 10 & 15 & 20 & 30 & 105 \\
T & Data provision & 15,972 & 23.46 & 9.33 & 10 & 18.75 & 23 & 26 & 119 \\
W & Administrative termination & 11 & 25.36 & 12.73 & 14 & 17.5 & 22 & 27 & 60 \\
X & Interruption for arrears & 705 & 31.56 & 20.63 & 10 & 17 & 26 & 39 & 117 \\
Z & Meter replacement & 233,143 & 26.73 & 19.94 & 10 & 14 & 21 & 29 & 119 \\
\hline
\caption{Descriptive statistics of intervention durations by activity type (in minutes).}
\label{tab:activity_statistics}
\end{longtable}

\endgroup

Table \ref{tab:activity_statistics} provides comprehensive descriptive statistics for intervention durations by activity type. Figure \ref{fig:violin} presents the distribution of work durations across different activity types, while Figure \ref{fig:bins_norm} shows the normalized binned distributions to highlight pattern similarities despite frequency differences

\begin{figure}[H]
    \centering
    \includegraphics[width=\textwidth]{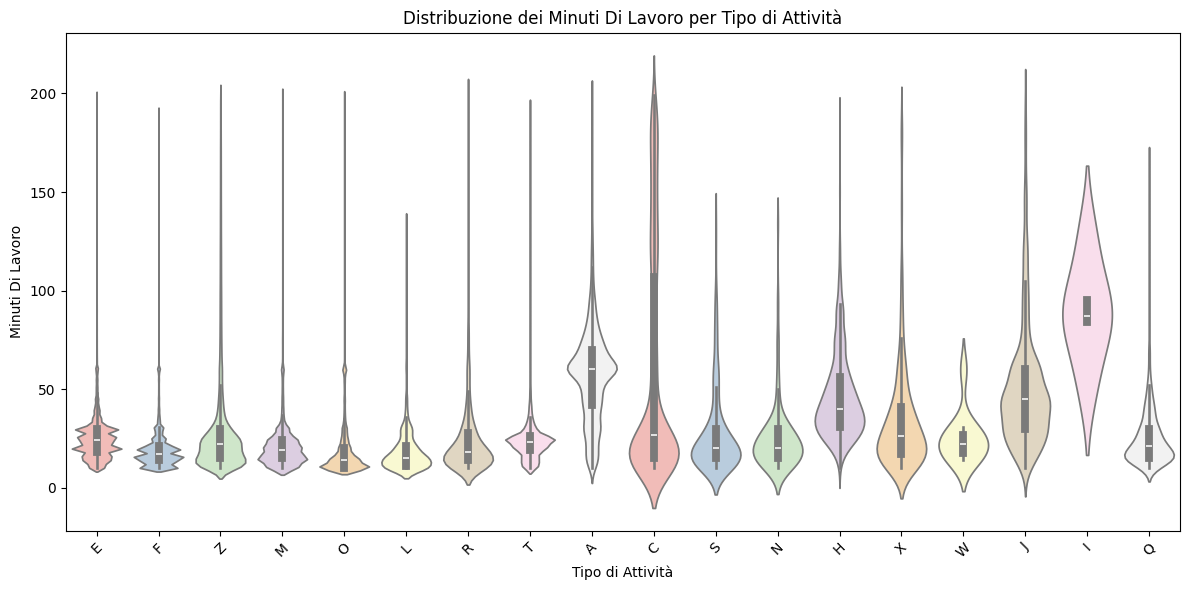}
    \caption{Distribution of work duration (minutes) by activity type.}
    \label{fig:violin}
\end{figure}

\begin{figure}[H]
    \centering
    \includegraphics[width=\textwidth]{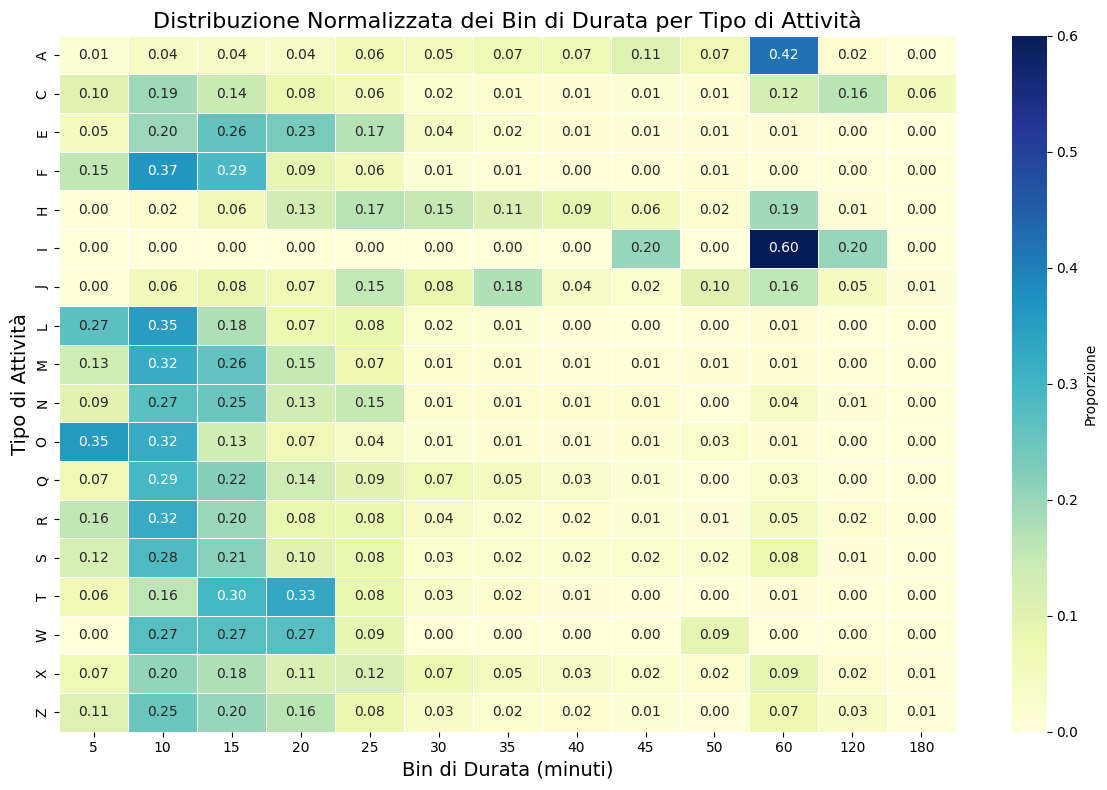}
    \caption{Normalized binned distribution of intervention durations by activity type.}
    \label{fig:bins_norm}
\end{figure}

\section*{List of Abbreviations}

\begin{table}[H]
\begin{tabular}{ll}
\toprule
\textbf{Abbreviation} & \textbf{Definition} \\
\midrule
AI & Artificial Intelligence \\
BF & Back-and-Forth \\
CDF & Cumulative Distribution Function \\
CI & Confidence Interval \\
CSO & Contextual Stochastic Optimization \\
CVRP & Capacitated Vehicle Routing Problem \\
CVRPSD & Capacitated VRP with Stochastic Demand \\
DoE & Design of Experiments \\
DRL & Deep Reinforcement Learning \\
DRSO & Distributionally Robust Stochastic Optimization \\
EA & Evolutionary Algorithm \\
EI & Expected Improvement \\
GES & Grouping Evolution Strategy \\
GNN & Graph Neural Network \\
GPS & Global Positioning System \\
GSO & Graph Shift Operators \\
KPI & Key Performance Indicator \\
MAE & Mean Absolute Error \\
MAPE & Mean Absolute Percentage Error \\
ML & Machine Learning \\
MOEA & Multi-Objective Evolutionary Algorithm \\
NSGA & Non-dominated Sorting Genetic Algorithm \\
OR & Operations Research \\
PDF & Probability Density Function \\
PR & Preventive Restocking \\
RMSE & Root Mean Squared Error \\
RRN & Recurrent Relational Network \\
SCO & Stochastic Contextual Optimization \\
SDI & Servizio Distribuzione Impianti (Gas Meter Service) \\
SMS-EMOA & S-Metric Selection Evolutionary Multi-objective Algorithm \\
SPO & Smart Predict-then-Optimize \\
SVRP & Stochastic Vehicle Routing Problem \\
TWCVRP & Time Window Capacitated Vehicle Routing Problem \\
VRP & Vehicle Routing Problem \\
VRPB & VRP with Backhauling \\
VRPPD & VRP with Pickups and Deliveries \\
\bottomrule
\end{tabular}
\end{table}


\end{document}